\documentclass[12pt,reqno]{amsart}
\usepackage{amsmath,amssymb,amsfonts,amsthm,mathrsfs, url}
\usepackage[colorlinks=true,linkcolor=blue]{hyperref}
\usepackage{graphicx}
\usepackage[usenames,dvipsnames]{color}
\usepackage{enumerate}
\usepackage[margin=1in]{geometry}
\usepackage[normalem]{ulem}
\usepackage{bbm}

\usepackage{tikz}
\usepackage{caption}
\usepackage{subcaption}

\usetikzlibrary{decorations.pathreplacing}
\usetikzlibrary{decorations.markings}
\newcommand{\x}{.5}

\newcommand{\y}{1}

\tikzstyle{vert}=[shape=circle,draw=black,fill=white, inner sep=.5mm]
\tikzstyle{arrow}= [thick,decoration={markings,mark=at position 0.6 with {\arrow{>}}}]

\theoremstyle{plain}
\newtheorem{thm}{Theorem}
\newtheorem{lem}[thm]{Lemma}
\newtheorem{cor}[thm]{Corollary}

\newtheorem{remark}[thm]{Remark}
\newtheorem{question}{Question}
\newtheorem{conj}{Conjecture}
\theoremstyle{definition}
\newtheorem{defn}{Definition}
\newtheorem{RE}{Random Experiment}
\theoremstyle{remark}
\newtheorem{claim}{Claim}

\newcommand{\rbrac}[1]{\left(#1\right)} 
\newcommand{\sbrac}[1]{\left[ #1\right]} 

\newcommand{\diam}{\operatorname{diam}}

\newcommand{\real}{\ensuremath {\mathbb R} }	
\newcommand{\ent}{\ensuremath {\mathbb Z} }
\newcommand{\nat}{\ensuremath {\mathbb N} }

\newcommand{\mbf}[1] {\text{\boldmath$#1$}}
\newcommand{\remove}[1] {}

\newcommand{\ex} {{\bf E}} 
\newcommand{\pr} {{\bf Pr}}

\newcommand{\cC} {\ensuremath{\mathcal C}}
\newcommand{\cE} {\ensuremath{\mathcal E}}
\newcommand{\cG} {\ensuremath{\mathcal G}}
\newcommand{\cH} {\ensuremath{\mathcal H}}
\newcommand{\cU} {\ensuremath{\mathcal U}}

\newcommand{\bX} {\ensuremath{\mbf X}}
\newcommand{\Qd} {\ensuremath{{\mathcal{Q}_d}}}
\newcommand{\Td} {\ensuremath{{\mathcal{T}_d}}}

\newcommand{\Tone} {\ensuremath{{\mathcal{T}_1}}}

\newcommand{\eps}{\varepsilon}

\newcommand{\suma}[1]{\sum_{\substack{#1}}}

\DeclareMathOperator{\dist}{dist}

\DeclareMathOperator{\Bin}{Bin}
\DeclareMathOperator{\vol}{vol}

\renewcommand{\l}{\ell}

\title[Almost-perfect packings and Tuza's conj. in the random geometric graph]{Almost-perfect packings and Tuza's conjecture in the random geometric graph}
\author{Patrick Bennett}
\address{Department of Mathematics, Western Michigan University, Kalamazoo, MI, USA}
\thanks{The first author was supported in part by Simons Foundation Grant \#848648.}
\email{\tt patrick.bennett@wmich.edu}

\author{Ryan Cushman}\thanks{}
\address{Department of Mathematics,
University of Tennessee, Knoxville, Knoxville, TN, USA} 
\email{\tt rcushma3@utk.edu}

\author{Andrzej Dudek}
\address{Department of Mathematics, Western Michigan University, Kalamazoo, MI, USA}
\thanks{The third author was supported in part by Simons Foundation Grant MPS-TSM-00007551.}
\email{\tt andrzej.dudek@wmich.edu}

\author{Xavier P\'erez-Gim\'enez}
\address{Department of Mathematics, University of Nebraska-Lincoln, Lincoln, NE, USA}
\thanks{The fourth author was supported in part by NSF grant DMS2201590.}
\email{\tt xperez@unl.edu}

\date{}

\begin{document}
\maketitle

\begin{abstract}
    The triangle packing number $\nu(G)$ of a graph $G$ is the maximum size of a set of edge-disjoint triangles in $G$. Tuza conjectured that in any graph $G$ there exists a set of at most $2\nu(G)$ edges intersecting every triangle in $G$. We show that Tuza's conjecture holds in the random geometric graph for a large range of densities. We also study the problem of covering almost all edges of the random geometric graph with edge-disjoint copies of some fixed graph $F$.
In particular, we show the existence of almost-perfect packings for an infinite family of $F$, and state some negative results as well.
\end{abstract}

\section{Introduction}
A {\em triangle packing} is a collection of edge-disjoint triangles in a graph. A triangle packing is {\em perfect} if it uses all of the edges in the graph.
For a graph $G$, the {\em triangle packing number} $\nu(G)$ is the size of the largest triangle packing in $G$. The {\em triangle cover number} $\tau(G)$ is the size of the smallest set of edges that covers all triangles in $G$. Clearly, $\nu(G)\le \tau(G)\le 3\nu(G)$. Tuza \cite{tuza1, tuza2} conjectured the following improved upper bound.
\begin{conj}[Tuza's Conjecture]\label{tuzaconj} 
 For every graph $G$, $\tau(G) \le 2 \nu(G)$.
\end{conj}
This conjecture has since attracted considerable attention and there are partial results and generalizations. The conjecture is tight for the complete graphs of orders 4 and 5. Baron and Kahn~\cite{BK} showed (disproving a conjecture of Yuster~\cite{Yuster}) that there are arbitrarily large graphs $G$ of edge density bounded away from zero satisfying $\tau(G) \ge (2-o(1))\nu(G)$. Hence, in general, the multiplicative constant 2 in the Tuza's conjecture cannot be improved.
 The best known general upper bound is due to Haxell~\cite{haxell}, who proved that $\tau(G) \le \frac{66}{23} \nu( G)$.  For more related results for deterministic graphs, the interested reader can consult~\cite{AZ, HR, krivelevich}. 

 Tuza's conjecture has also been studied for random graphs, and in particular the conjecture is known to be true for ``almost all'' graphs. We say that an event $\cE(n)$ happens \textit{asymptotically almost surely (a.a.s.)} if $\pr[\cE(n)] \rightarrow 1$ as $n \rightarrow \infty$.  Kahn and Park \cite{KP} proved that for every function $p=p(n) \in [0, 1]$, a.a.s.\ the Erd\H{o}s-R\'enyi random graph $G(n, p)$ satisfies Tuza's Conjecture. Independently the first, second and third authors~\cite{BCD} showed that for every $m \le \binom n2$, a.a.s.\ the random graph $G(n, m)$ satisfies Tuza's conjecture. This improved upon previous work by the first and third authors and Zerbib~\cite{BDZ}, who proved that Tuza's conjecture holds a.a.s.\ in $G(n, m)$ except for possibly some values $m = \Theta(n^{3/2})$.

In this note we investigate Tuza's conjecture for random geometric graphs (defined in the next paragraph). 
We also investigate the more general problem of finding large packings of a fixed graph $F$ in a random geometric graph. Here an {\em $F$-packing} (or {\em packing}, for short) denotes a collection of edge-disjoint copies of $F$. We say an $F$-packing in $G$ is perfect if it has size $m_G/m_F$ (where $m_G$ denotes the number of edges in $G$).
In analogy to the triangle case, the {\em $F$-packing number} $\nu_F(G)$ of a graph $G$ is the size of the largest $F$-packing in $G$. When $F$ is the $t$-clique $K_t$, we simply write $\nu_t(G)=\nu_{K_t}(G)$. When $F$ is a triangle, we omit the subscript and write $\nu(G)$.

Let $\Qd=[0,1]^d$ be the unit $d$-cube.
Given a set of points $\bX \subseteq \Qd$ and $r\ge0$, the {\em geometric graph $\cG_{\Qd}(\bX,r)$ on the cube} is the graph with vertex set $\bX$ and with any two points $X,Y\in\bX$ connected by an edge if their Euclidean distance $\| X-Y\|_2$ is at most $r$. We also consider a toroidal version of this model. Let $\Td=\Qd/\sim$ be the unit $d$-torus resulting from identifying any two points in the cube $\Qd$ whose coordinates differ by an integer. We equip $\Td$ with the usual toroidal distance
\[
\dist_\Td(X,Y) = \min_{Z\in\ent^d} \| X-Y+Z\|_2.
\]
The {\em geometric graph $\cG_{\Td}(\bX,r)$ on the torus} is defined analogously to $\cG_{\Qd}(\bX,r)$ but with $\bX \subseteq \Td$ and by connecting any two points $X,Y\in\bX$ by an edge if $\dist_\Td(X,Y)\le r$.
Note that $\cG_{\Qd}(\bX,r) \subseteq \cG_{\Td}(\bX,r)$ since the latter may contain additional edges that go ``around'' the torus (or, more precisely, since $\dist_\Td(X,Y) \le \|X-Y\|_2$).
We are interested in the case where $\bX$ is random. More precisely, let $\bX=\{X_1,\ldots,X_n\}$ where $X_1,\ldots,X_n \in [0,1]^d$ are chosen uniformly at random and independently.
We will interpret $\bX$ as a random subset of $\Qd$ or $\Td$ depending on the context (with the mild abuse of notation of equating $[0,1]^d$ and $[0,1]^d/\sim$ in the toroidal case). Note that, in either case, $\bX$ consists of $n$ distinct points with probability one. For random $\bX$ and $*\in\{\Qd,\Td\}$, we call  $\cG_*(\bX,r)$ the {\em random geometric graph} ({\em on the cube} for $*=\Qd$ and {\em on the torus} for $*=\Td$).

For the classical Erd\H os-R\'enyi random graph models $G(n, p)$ and $G(n, m)$, the aforementioned authors in~\cite{ BCD,BDZ, KP} found that for sparse enough graphs, almost all triangles are edge disjoint, meaning that $\tau \approx \nu$. This situation persists until we have $\Theta(n^{3/2})$ edges. Once we have much more than $n^{3/2}$ edges, we get the following different situation. First, $\tau$ will be about half the number of edges $m$. Of course $\tau$ is always at most $\frac 12 m$ since we can make the graph bipartite (and therefore triangle-free) by removing at most $\frac 12 m$ edges, and when $m \gg n^{3/2}$ we find that one cannot do much better. Meanwhile we always have $\nu \le \frac 13 m$ since every triangle uses 3 edges. When $m \gg n^{3/2}$ we have that every edge is in many triangles, at which point we can find a triangle packing using almost all the edges, meaning that $\nu \approx \frac13 m$. As a result, when $m \gg n^{3/2}$ we have that $\tau$ is about $\frac 32 \nu$. Thus, for most densities the random graphs $G(n, p)$ and $G(n, m)$ actually satisfy Tuza's conjecture with some room to spare, and the more interesting density is when $m = \Theta(n^{3/2})$. 

For the random geometric graph we will observe a similar phenomenon. When the graph is sparse enough, we will have $\tau \approx \nu$ since almost all triangles are edge-disjoint. When the graph is dense enough, we can cover almost all edges and so $\nu \approx \frac{m}{3}$. Thus $\tau \le \frac{m}{2} \approx \frac{3}{2} \nu$. For the random geometric graph, the interesting ``in-between'' range is when $r= \Theta(n^{-1/d})$, and we will show that a.a.s.\ Tuza's conjecture holds for the random geometric graph except for possibly when $r$ is in that range (but we avoid discussing the case where $r$ is more than 1/2).

Our first theorem estimates $\nu$ in the sparse regime. Of course, we always have $\nu(G) \le t$ for any graph $G$ with $t$ triangles.

\begin{thm}\label{thm:sparse}
For each $d \ge 1$, $*\in\{\Qd,\Td\}$ and real $\delta > 0$ there is a constant $\kappa_1=\kappa_1(d, \delta)>0$  so that the following holds. For $ r < \kappa_1n^{-1/d}  $ we have that a.a.s. 
\[
\nu( \cG_{*}(\bX,r) ) \ge (1-\delta)t(\cG(\bX,r)),
\]
where $t(\cG_*(\bX,r))$ is the number of triangles of $\cG_*(\bX,r)$.
\end{thm}
In particular, for $r=o(n^{-1/d})$, a.a.s.\ $\nu( \cG_*(\bX,r) ) \sim t(\cG_*(\bX,r)).$

 Our next theorem estimates $\nu$ in the dense regime. The theorem says that the upper bound (discussed above) $\nu \le \frac13 m$ is close to the truth in the dense regime.

\begin{thm}\label{thm:dense}
For each $d \ge 1$, $*\in\{\Qd,\Td\}$ and real $\delta > 0$ there are positive constants $\kappa_2=\kappa_2(d, \delta)$, $\kappa_3=\kappa_3(d, \delta)$ (for $*=\Td$, we can pick $\kappa_3=1/2$) so that the following holds. For any $r$ with $\kappa_2n^{-1/d} \le r \le \kappa_3,$ we have that a.a.s.
\[
\nu( \cG_{{*}}(\bX,r) ) \ge (1 - \delta) \frac13 m{(\cG_*(\bX,r))},
\]
where $m(\cG_*(\bX,r))$ is the number of edges of $\cG_*(\bX,r)$.
\end{thm}
In particular, for $r=\omega(n^{-1/d})$ and $r=o(1)$, a.a.s.\ $\nu( \cG_{{*}}(\bX,r) ) \sim \frac13 m(\cG_*(\bX,r))$, so we obtain an almost-perfect triangle packing. For the toroidal case, when $*=\Td$, we may choose $\kappa_3=1/2$ for every $d$ and $\delta$. However, for $*=\Qd$, $\kappa_3$ depends on $\delta$ (and $d$) since our argument ignores a $\Theta(dr)$ fraction of the edges that are close to the boundaries of $[0,1]^d$.

In his Bachelor's thesis, Burggraf \cite{Burggraf} recently established a version of Theorem~\ref{thm:dense} for the case $*=\mathcal{Q}_2$ (a result that was unknown to us during the preparation of this manuscript). More generally, he investigated the problem of finding $C_k$-packings in $\mathcal{Q}_2$ for fixed $k$.

Note that if we set $\delta = 1/2$ in Theorem~\ref{thm:sparse} then we get
\[
\nu(\cG_{*}(\bX, r)) \ge t(\cG_*(\bX,r))/2 \ge \tau(\cG_{*}(\bX, r))/2.
\]
Likewise if we set $\delta = 1/4$ in Theorem~\ref{thm:dense} then we get
\[
\nu(\cG_*(\bX, r)) \ge m{(\cG_*(\bX,r))}/4 \ge \tau(\cG_{*}(\bX, r))/2.
\]
Thus we have the following corollary. 
\begin{cor}\label{cor:Tuza}
    For each $d \ge 1$ and $*\in\{\Qd,\Td\}$, there exist {positive} constants $\kappa_1 = \kappa_1(d, 1/2)$, $\kappa_2=\kappa_2(d, 1/4)$ and $\kappa_3=\kappa_3(d, 1/4)$ such that for any $r \in [0, \kappa_3] \setminus [\kappa_1n^{-1/d}, \kappa_2n^{-1/d}]$ we have a.a.s.\ that 
   \[
  \tau(\cG_{*}(\bX, r)) \le 2 \nu(\cG_{*}(\bX, r)).
   \]
   In other words, Tuza's conjecture holds for $\cG_{*}(\bX,r)$.
\end{cor}

The cases $d=1$ and $d\ge2$ of Theorem~\ref{thm:dense} have significant (and interesting) differences. As we were working on the proof,  we realized it was important to pay attention to the (Euclidean) length of the edges in the triangles in our packing. The lengths of edges in the random geometric graph are randomly distributed on $[0, r]$, but for $d \ge 2$ this is not a uniform distribution since longer edges are more likely than shorter edges. So when making a triangle packing, it is dangerous to use triangles that have both long edges and short edges, since we might run out of short edges before we use all our long edges.
In view of this, for $d\ge2$, we had the idea to pack nearly-equilateral triangles, so edges get matched with other edges of similar lengths.
For $d=1$, our idea of packing nearly-equilateral triangles does not work, so we had to come up with a different construction that packs triangles with very different edge lengths.

Moreover, our argument for the $d\ge2$ case inspired some further generalization. Observe that, for $d\ge2$, a triangle can be drawn in $\mathbb{R}^d$ with all edges of the same strictly positive length. (This is not possible for $d=1$.)
Any graph with that same property is called a {\em distance graph} in $\mathbb{R}^d$ (see the precise definition below). It turns out that Theorem~\ref{thm:dense} still holds if we replace triangles by graphs in a large family of graphs that contains distance graphs.
Recall that a graph $G$ is a {\em blowup} of a graph $H$ if $G$ is formed by duplicating each vertex of $H$ some number of times (where we do not require that every vertex is duplicated the same number of times). 

\begin{defn}
    For $d \ge 1$, a distance graph in $\real^d$ is a finite graph whose vertices can be injectively mapped to $\real^d$ so that every pair of adjacent vertices lie at the same strictly positive Euclidean distance. Let $\cU(d)$ be the set of all graphs $F$ such that $F \subseteq F'$, where $F'$ is a blowup of a distance graph in $\mathbb{R}^d$.
\end{defn}
Equivalently, we say $F \in \cU(d)$ if we can draw $F$ in $\mathbb{R}^d$, possibly allowing multiple vertices to occupy the same point, such that all edges of $F$ have the same Euclidean positive length (say, length $1$). We are able to generalize Theorem~\ref{thm:dense} to all graphs in $\cU(d)$. This proves a conjecture of Burggraf \cite{Burggraf}.

\begin{thm}\label{thm:unitdistance}
Let $d\ge1$, $*\in\{\Qd,\Td\}$, $F\in \cU(d)$ and $\delta >0$. Then, there exist positive constants $\kappa_2=\kappa_2(F,d, \delta)$ and $\kappa_3 = \kappa_3(F, d,\delta)$ (with $\kappa_3=\kappa_3(F, d)$ for $*=\Td$) so that the following holds. For any $r$ with $\kappa_2n^{-1/d} \le r \le \kappa_3$ we have that a.a.s.
\[
\nu_F( \cG_{*}(\bX,r) ) \ge (1 - \delta) \frac m{m_F} ,
\]
where $m$ and $m_F$ are the number of edges of $\cG_{*}(\bX,r)$ and $F$, respectively.  
\end{thm}

As an immediate corollary, since the clique $K_t$ is in $\cU(d)$ for $t\le d+1$, we get:
\begin{cor}\label{cor:cliques}
Let $d\ge1$, $*\in\{\Qd,\Td\}$, $t\le d+1$ and $\delta >0$. Then, there exist positive constants $\kappa_2=\kappa_2(t,d, \delta)$ and $\kappa_3 = \kappa_3(t, d,\delta)$ (with $\kappa_3=\kappa_3(t, d)$ not depending on $\delta$ when $*=\Td$) so that the following holds. For any $r$ with $\kappa_2n^{-1/d} \le r \le \kappa_3$ we have that a.a.s.
\[
\nu_t( \cG_*(\bX,r) ) \ge (1 - \delta) \frac {m}{\binom{t}{2}},
\]
where $m$ is the number of edges in $\cG_*(\bX,r)$.
\end{cor}

Note that Theorem~\ref{thm:unitdistance} implies the case $d\ge2$ of Theorem~\ref{thm:dense}, since triangles belong to $\cU(d)$ for $d\ge2$. We believe the conclusion of Theorem~\ref{thm:unitdistance} should hold in many other cases where $F \notin \cU(d)$. One example is triangles in dimension $d=1$, which holds in view of Theorem~\ref{thm:dense} even if $K_3\notin\cU(1)$. We believe that Theorem~\ref{thm:unitdistance} should still hold if one can draw $F$ in $\mathbb{R}^d$ so that, with only one exception, all edges have the same length and the exceptional edge is just a little bit shorter than the rest. However, Theorem~\ref{thm:unitdistance} cannot be extended to arbitrary $F, d$. Indeed, we prove the following negative result, which asserts the nonexistence of almost-perfect packings of large cliques in low-dimensional random geometric graphs. This is essentially because for any way of drawing a large clique in a low dimensional space such that all pairs of vertices have distance at most $r$, there must be many pairs of vertices at distance much smaller than $r$. Thus, when one tries to pack these cliques, one always runs out of short edges long before one can use all the long edges.

\begin{thm}\label{thm:neg}
Let $d \ge 1$, $*\in\{\Qd,\Td\}$ and $\delta>0$. Then there exist positive constants $\kappa_3=\kappa_3(d,\delta)$ (with $\kappa_3=1/2$ for $*=\Td$) and $t_0=t_0(d,\delta)$ such that the following holds. For any $t\ge t_0$ and $r=r(n) \le \kappa_3$, we have that a.a.s.
\[
\nu_t(\cG_*(\bX,r)) \le (1+\delta) \rho_d \cdot \frac{m(\cG_*(\bX,r))}{\binom t2},
\]
where $\rho_d$ is a constant defined in Section~\ref{sec:neg} that satisfies $\rho_d \le 1/2$ and $\rho_d \le \sqrt{d} (3/4)^{d/2}$ for all $d\ge1$.
\end{thm}
In other words, for sufficiently large $t$ given the dimension $d$, the largest $K_t$-packing uses at most roughly half of the edges of the random geometric graph (or exponentially many fewer if $d$ is large). This is in striking contrast with Corollary~\ref{cor:cliques}, which guarantees almost-perfect $K_t$-packings for the case when $t\le d+1$.

\smallskip

\paragraph{{\bf Paper organization:}}
In Section~\ref{sec:triangles} we focus on triangles and prove Theorems~\ref{thm:sparse} and~\ref{thm:dense} (the latter only for $d\ge2$). In Section~\ref{sec:general}, we extend these ideas to packings of general graphs in $\cU(d)$ and prove Theorem~\ref{thm:unitdistance}. In Section~\ref{sec:d1}, we complete the proof of Theorem~\ref{thm:dense} by considering the case $d=1$. Note that our triangles are no longer near-equilateral there. In Section~\ref{sec:neg} we prove Theorem~\ref{thm:neg}, which provides a negative counterpart to Corollary~\ref{cor:cliques}. Finally, we include some concluding remarks and open problems in Section~\ref{sec:final}.

\section{Tuza's conjecture and Theorems~\ref{thm:sparse} and~\ref{thm:dense} (for $d\ge2$)} \label{sec:triangles}
\begin{proof}[Proof of Theorem~\ref{thm:sparse}]
Let $Y (=t)$ be the number of triangles in $\cG_{*}(\bX,r).$ We apply the second moment method to estimate $Y$. Let $p$ be the probability that three particular distinct random points $X_i,X_j,X_k\in\bX$ form a triangle, so $\ex[Y] = \binom n3 p$. Consider the toroidal case first, that is $*=\Td$. If $X_j,X_k$ are both within distance $r/2$ of $X_i$ then we always have the triangle $X_iX_jX_k$. Thus $p \ge [\theta (r/2)^d]^2$, where $\theta=\theta_d$ is the volume of a unit ball in $\real^d$.
For the cubic case, $*=\Qd$, we obtain $p \ge (1-r)^d[\theta (r/2)^d]^2 \ge (1+o(1))[\theta (r/2)^d]^2$ by further imposing that $X_i$ is at least distance $r/2$ from all the facets of the cube.
Now in order to have any chance of forming a triangle, $X_j$ and $X_k$ must be within distance $r$ of $X_i$, and so $p \le  [\theta r^d]^2$. Thus we have $\ex[Y] = \Theta(n^3 r^{2d})$ for both the toroidal and the cubic models. Now we bound $\ex[Y^2].$ For each triple $ijk$ of distinct indices in $[n]$, let $\cE(ijk)$ be the event that $X_iX_jX_k$ form a triangle. Then (explanation follows)
\begin{align*}
    \ex[Y^2] & = \ex \sbrac{\rbrac{\sum_{ijk} \mathbbm{1}_{\cE(ijk)}}^2}\\
    & = \sum_{ijk, i'j'k'} \pr[\cE(ijk) \wedge \cE(i'j'k')]\\
    & = O(n^3 r^{2d} + n^4 r^{3d} + n^5r^{4d}) + \binom n3 p \cdot \binom {n-3}{3} p\\
    & = (1+o(1))\ex[Y]^2.
\end{align*}
The sum on the second line is over all (ordered) pairs of (unordered) triples of distinct indices in $[n]$. On the third line, the first big-O term accounts for the terms where $ijk=i'j'k'$. The second big-O term takes care of the terms where our triples share two indices, and the last big-O term is for when the triples share one index. The main term on the third line is for disjoint triples. Now by the second moment method, a.a.s.\ we have $Y = (1+o(1))\ex[Y]$.

Let $Z$ be the number of pairs of triangles sharing exactly one edge, i.e.~ the number of copies of the graph $K_4^-$ formed by removing one edge from $K_4$. Similarly to our treatment of $Y$, one can use the second moment method to show that a.a.s.\ $Z = (1+o(1))\ex[Z] = \Theta(n^4 r^{3d}) = \Theta(nr^d Y).$ Since $nr^d \le \kappa_1^d$, we can choose $\kappa_1$ small enough so that $Z < \delta Y$. Thus there are at least $(1-\delta)Y = (1-\delta)t$ triangles which do not share any edge with another triangle. This completes the proof of Theorem~\ref{thm:sparse}.
\end{proof}




Next we show Theorem~\ref{thm:dense} for $d\ge2$. This is a particular case of the more general Theorem~\ref{thm:unitdistance}, so we treat this as a warm-up for the next proof. Here we will introduce some facts and tools that we will also use in the proof of Theorem~\ref{thm:unitdistance}.

\begin{proof}[Proof of Theorem~\ref{thm:dense} ($d\ge 2$)]

Most of the proof will apply to the $d=1$ case as well, so in general we assume $d\ge 1$. We will indicate the parts that require $d\ge2$ and then in Section~\ref{sec:d1} we will explain how to modify these parts to address the $d=1$ case.

Our goal is to find a collection of edge-disjoint triangles that uses almost all edges of the random geometric graph $\cG_*(\bX,r)$. Moreover, these triangles will all be ``nearly equilateral'' (for $d\ge2$ only).

It is well known that the number of edges $m$ of $\cG_{\Qd}(\bX,r)$ or $\cG_{\Td}(\bX,r)$ is typically around $\theta r^d n^2/2$. Below we include a proof of this fact in a precise form that is convenient for our argument. Similar versions of this statement may be found in the literature. (For instance, Chapter~3 of~\cite{Penrose} proves several much more general central-limit-type results for the fixed-size subgraph counts in $\cG_{\Td}(\bX,r)$, from which one can easily derive concentration of the number of edges.)
\begin{lem}\label{lem:edges}
Let $d\ge1$ and $1/2 \ge r=\omega(n^{-2/d})$. Then, a.a.s., $\cG_{\Td}(\bX,r)$ has $(\theta/2+o(1))n^2 r^d$ edges, and all but at most a $(2d+o(1))r$ fraction of those are also edges of $\cG_{\Qd}(\bX,r)$.
\end{lem}
\begin{proof}
We will use two simple second moment arguments akin to those in the proof of Theorem~\ref{thm:sparse}.
Let $Y$ be the number of edges in $\cG_{\Td}(\bX,r)$. The probability that two different vertices $X_i,X_j$ in $\bX$ are adjacent in $\cG_{\Td}(\bX,r)$ is exactly $\theta r^d$ (since $r\le 1/2$ and thus the ball of radius $r$ around $X_i$ does not wrap around the torus). Hence, $\ex Y = \binom{n}{2}\theta r^d \sim (\theta/2)n^2 r^d\to\infty$, since $r=\omega(n^{-2/d})$. Moreover (explanation follows),
\begin{equation}
\label{eq:Y2moment}
\ex Y^2 = O(n^2r^d+n^3r^{2d}) + \binom{n}{2}\binom{n-2}{2}\theta^2r^{2d} \sim (\ex Y)^2.
\end{equation}
Here, the two big-O terms account for the number of pairs of edges with two or one common endpoints. The last term in $\ex Y^2$ is the number of (ordered) pairs of edges with distinct endpoints.
By the second moment method, a.a.s.\ $Y=(1+o(1))\ex Y = (\theta/2+o(1))n^2 r^d$.

An edge in $E(\cG_{\Td}(\bX,r))\setminus E(\cG_{\Qd}(\bX,r)$) must have both endpoints within Euclidean distance $r$ from the boundary of $[0,1]^d$. Let $Z$ count the number of edges $X_iX_j$ ($i<j$) of $\cG_{\Td}(\bX,r)$ where $X_i$ is within distance $r$ of some facet of $[0,1]^d$. Clearly, $Z$ is an upper-bound on $|E(\cG_{\Td}(\bX,r))\setminus E(\cG_{\Qd}(\bX,r))|$. In a similar way to $Y$, $\ex Z = \binom{n}{2}(1-(1-2dr)^d)\theta r^d \le 2dr \ex Y$.
Moreover,
\[
\ex Z^2 = O(dr n^2r^d+dr n^3r^{2d}) + \binom{n}{2}\binom{n-2}{2} (1-(1-2dr)^d)^2 \theta^2r^{2d} \sim (\ex Z)^2,
\]
and a second moment argument again gives that a.a.s.~$Z\sim \ex Z \le 2dr \ex Y \sim 2dr (\theta/2)n^2r^d$. This finishes the proof of the lemma.
\end{proof}
\begin{remark}\label{rem:torus}
In view of Lemma~\ref{lem:edges}, the conclusion of Theorem~\ref{thm:dense} for the cubic model ($*=\Qd$) will follow from that of the toroidal model ($*=\Td$). Indeed, suppose that one can find a collection of at least $(1-\delta/2)m(\cG_{\Td}(\bX,r))/3$ edge-disjoint triangles in $\cG_{\Td}(\bX,r)$. After removing all the triangles from the collection that are not contained in $\cG_{\Qd}(\bX,r)$, we still have at least
\begin{equation*}
(1-\delta/2 - (6d+o(1))r )m(\cG_{\Td}(\bX,r))/3
\end{equation*}
triangles left in $\cG_{\Qd}(\bX,r)$. Picking constant $\kappa_3=\kappa_3(d,\delta)$ to be sufficiently small, we still have at least $(1-\delta)m(\cG_{\Td}(\bX,r))/3 \ge (1-\delta)m(\cG_{\Qd}(\bX,r))/3$ triangles, as desired.
Therefore, it will suffice to prove Theorem~\ref{thm:dense} for $*=\Td$.
\end{remark}

For the remaining of this section, we assume $*=\Td$ and interpret all distances and geometric descriptions to be in the torus. For simplicity, we often equate $[0,1]^d$ to $\Td$ by implicitly identifying points at toroidal distance $0$. Fix $\kappa_3=1/2$. In particular, the ball of radius $r$ does not wrap around the torus and has volume exactly $\theta r^d$, where recall $\theta$ is the volume of a unit ball in $\real^d$.


A {\em hypergraph} $\cH$ on a vertex set $V=V(\cH)$ is a collection of subsets of $V$. The elements of $\cH$ are called {\em edges}.  We say $\cH$ is $\l$-bounded if every edge has size at most $\l$. A {\it matching} in $\cH$ is a set of pairwise disjoint edges. The {\it matching number} of $\cH$, denoted by $\nu(\cH)$, is the maximum possible size of a matching in $\cH$. We say that a function $f:\cH\rightarrow [0,1]$ is a {\it fractional matching} for $\cH$ if for every $v\in V(\cH)$, 
$$\sum_{e\in \cH:~v\in e}f(e)\le 1.$$ 
For a fractional matching $f$ for $\cH$, let $f(\cH):=\sum_{e\in \cH}f(e)$ and
$$\alpha(f):=\max_{x\neq y \in V(\cH)} \sum_{e\in \cH:~x,y\in e} f(e).$$

Kahn \cite{K96} showed that under certain conditions, a fractional matching can be used to obtain a matching of nearly the same size:
\begin{lem}[Theorem 1.2 in \cite{K96}]\label{lem:Kahn}
    For every integer $\l\ge 1$ and every real $\eps>0$, there exists a real $\sigma>0$ such that whenever $\cH$ is an $\l$-bounded hypergraph and $f$ is a fractional matching of $\cH$ with $\alpha(f)<\sigma$, then $\nu(\cH)>(1-\eps)f(\cH).$
\end{lem}

We fix some small $\eps>0$. (Our choice of $\eps$ will depend on $\delta$ and $d$. In turn, $\kappa_2$ will be sufficiently large given $\eps$, $\delta$ and $d$, but recall $\kappa_3=1/2$ for the toroidal case that we are considering.)
We partition the torus $[0,1]^d$ into $n_0$ cubic cells of side length $s=\lceil1/\eps r\rceil^{-1}$, so in particular $\eps r/(1+\eps r) \le s\le \eps r$ and $n_0=s^{-d}$ is an integer. From our assumptions on $r$, we have
\begin{equation}\label{eq:n0range}
1/(\eps \kappa_3)^d \le n_0 \le  \left(\frac{1+\eps \kappa_3}{\eps \kappa_2}\right)^d n \le \left(\frac{2}{\eps \kappa_2}\right)^d n.
\end{equation}
Henceforth, we will assume that each point in $\bX$ lies in the interior of a unique cell, which happens with probability $1$.  We define an auxiliary graph $\cC$ which we will call the {\em graph of cells}. Its vertex set $V(\cC)$ will be the set of $n_0$ cells, and two distinct vertices (cells) in $\cC$ are adjacent if the centers of the two cells have  distance at most $(1-d\eps)r$ (see Figure~\ref{fig:cells}).
By construction, any two points $X,Y\in[0,1]^d$ that lie in adjacent cells (or in the same cell) must be within (toroidal) distance $r$.
Note that $\cC$ is a regular graph with the degree of each cell at most
\begin{equation}\label{eq:degC}
\theta r^d / s^d = \theta r^d n_0.
\end{equation}
A triangle in $\cC$ is a set $\{C_1,C_2,C_3\}$ of distinct pairwise-adjacent cells.

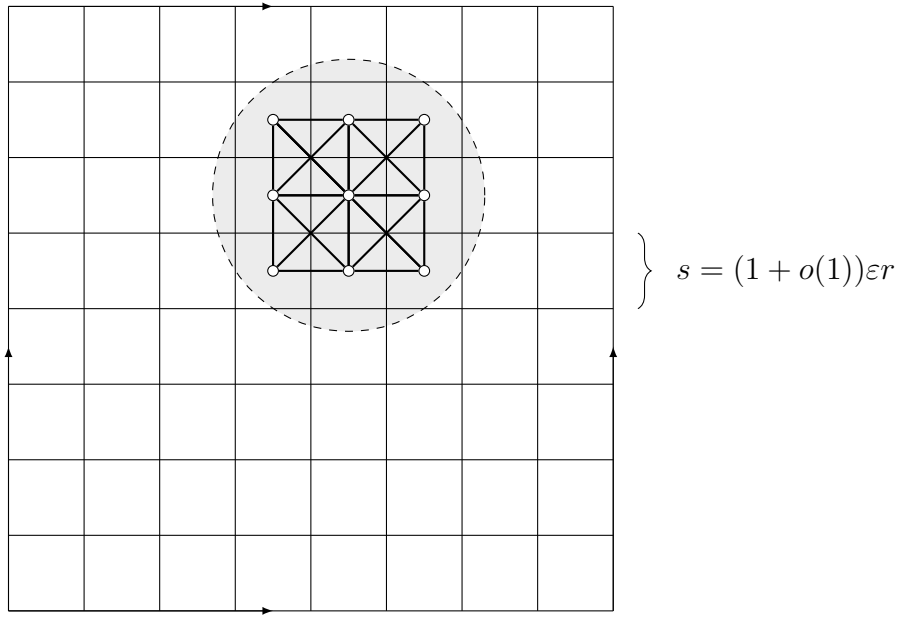
\begin{figure}
\centering
\begin{tikzpicture}[scale=1] 
\tikzstyle{vert}=[shape=circle,draw=black,fill=white, inner sep=.5mm]
\usetikzlibrary{decorations.pathreplacing}
\begin{scope}
\draw[dashed,fill=gray!15] (5.5,6.5) circle (1.8);
	\foreach \x in {1,2,3,4,5,6,7,8,9}
		{
    		\draw (\x,1) -- (\x,{9});
		\draw (1,\x) -- (9,\x);
		}
\end{scope}
\draw[-{latex}] (1,1) -- (4.5,1);
\draw[-{latex}] (1,1) -- (1,4.5);
\draw[-{latex}] (1,9) -- (4.5,9);
\draw[-{latex}] (9,1) -- (9,4.5);
\foreach \i/\x in {1/4.5,2/5.5,3/6.5}
	{
	\foreach \j/\y in {1/5.5,2/6.5,3/7.5} {
		\node[vert] (\i\j) at (\x,\y) {};
		}
	}
\foreach \x/\y in {22/11,22/12,22/13,22/32,22/23,22/33,22/13,22/31,22/21} {
	\draw[thick] (\x) -- (\y);
	}
\foreach \x/\y in {2/3,2/1} {
	\draw[thick] (\y\y) -- (\x\y);
	\draw[thick]  (\y\y) -- (\y\x);
	\draw[thick]  (\x\y) -- (\y\x);
	}	
\foreach \x/\y in {13/23,13/12,12/22,22/23,13/22,12/23} {
	\draw[thick] (\x) -- (\y);
	}	
\foreach \x/\y in {22/32,22/31,22/21,21/31,31/32,21/32} {
	\draw[thick] (\x) -- (\y);
	}					
\draw [decorate,decoration={brace,amplitude=5pt,mirror,raise=4ex}]
  (8.6,5) -- (8.6,6) node[midway,xshift=6.5em]{$s=(1+o(1))\varepsilon r$};
\end{tikzpicture}
\caption{The graph of cells $\cC$ for $d=2$.}
\label{fig:cells}
\end{figure}

Observe that each cell of $\cC$ contains $n/n_0$ points of $\bX$ in expectation. Next, we show that we can slightly perturb $\bX$ to guarantee that every cell has the same amount of points.
\begin{lem}\label{lem:trimming}Let $0<\eps<1$ and $\kappa_2>0$ be sufficiently large given $\eps$.
    A.a.s.\ it is possible to alter $\bX$ by adding at most $4\eps n$ points and removing at most $4\eps n$ points to form a new set of points $\bX'$ which has exactly $N:=\lceil n/n_0\rceil$ points in each cell.
\end{lem}

To prove the claim we will use the following standard bound.
\begin{lem}[Chernoff--Hoeffding bound, as stated in~\cite{JLR}] \label{lem:chernoff}
    Let $X$ be distributed as $\Bin(m, s)$. Then, for any $\beta>0$,
     \[
    \pr(X\ge ms+\beta) \le \exp\rbrac{-\frac{\beta^2}{2(ms + \beta/3)}}
    \]
    and
    \[
    \pr(X\le ms-\beta) \le  \exp\rbrac{-\frac{\beta^2}{2ms}}.
    \]
   \end{lem}

\begin{proof}[Proof of Lemma~\ref{lem:trimming}]
The number of points of $\bX$ that land in a given cell is $\Bin(n,1/n_0)$. We say a cell is {\em sparse} if it contains fewer than $(1-\eps)N$ points. From our definition of $N$, the upper bound in~\eqref{eq:n0range} and our assumption on $\kappa_2$ to be sufficiently large given $\eps$,
\begin{equation}\label{eq:Nbounds}
n/n_0 \le N < (1 + n_0/n) n/n_0 < (1 + \eps/2) n/n_0,
\end{equation}
so each sparse cell must contain fewer than $(1-\eps)(1 + \eps/2) n/n_0 < (1-\eps/2)n/n_0$ points.
Note that, for any set of cells, the events that require that each cell in the set is sparse are negatively correlated (this is a standard Balls-and-Bins fact; see~e.g.~\cite{DR88}).
In view of this and using the Chernoff bound from Lemma~\ref{lem:chernoff}, the probability that there are at least $\eps n_0$ sparse cells is at most 
\begin{align}
    \binom{n_0}{\lceil\eps n_0\rceil} \exp\rbrac{-\frac{((\eps/2) n/n_0)^2}{2n/n_0}\cdot \lceil\eps n_0\rceil} &\le \rbrac{\frac{e n_0}{\lceil\eps n_0\rceil}}^{\lceil \eps n_0\rceil} \exp\rbrac{-(\eps^2 /8) n/n_0 \cdot \lceil\eps n_0\rceil}\notag\\
    & \le \rbrac{\frac{e}{\eps} \exp\rbrac{-(\eps^2/8) n/n_0}}^{\eps n_0}
    \notag\\
    &= e^{-\Omega(n)}=o(1) .\label{eq:sparse}
\end{align}
To justify the exponential decay in the last step, note that the previous upper-bound is increasing in $n_0$ (as easily shown by taking the derivative) and it is $e^{-\Omega(n)}=o(1)$ for any $n_0\sim (2/\eps\kappa_2)^d n$ as long as $\kappa_2$ is sufficiently large given $\eps$. Hence, it is $o(1)$ for any $n_0$ in the range described in~\eqref{eq:n0range}.
Then, a.a.s.\ there are at most $\eps n_0$ sparse cells.
Thus we can then make all sparse cells contain exactly $N$ points by adding at most $N \eps n_0 \le (1+\eps/2) \eps n < 4\eps n$ points in total.

The analysis of the cells with more than $N$ points is slightly more complicated, since the more points we have in a cell, the fewer of those cells we can allow.
We say a cell is $0$-{\em dense} if it contains at least $(1+\eps)n/n_0$ points but less than $2n/n_0$.
Now for $i=1,2,3,\ldots$, a cell is $i$-{\em dense} if it contains at least $(i+1)n/n_0$ points but less than $(i+2)n/n_0$. 
By a calculation analogous to that in~\eqref{eq:sparse}, but using the upper-tail of the Chernoff bound in Lemma~\ref{lem:chernoff}, the probability that there are at least $\eps n_0$ 0-dense cells is at most
\begin{align*}
    \binom{n_0}{\lceil\eps n_0\rceil} \exp\rbrac{-\frac{(\eps n/n_0)^2}{2(1+\eps/3)n/n_0}\cdot \lceil\eps n_0\rceil} &\le \rbrac{\frac{e n_0}{\lceil\eps n_0\rceil}}^{\lceil \eps n_0\rceil} \exp\rbrac{-(\eps^2 /3) n/n_0 \cdot \lceil\eps n_0\rceil}\\
    & \le \rbrac{\frac{e}{\eps} \exp\rbrac{-(\eps^2/3) n/n_0}}^{\eps n_0}\\
    & = e^{-\Omega(n)} = o(1),
\end{align*}
for sufficiently large $\kappa_2$ given $\eps$ (where the exponential decay is justified as in~\eqref{eq:sparse}).
By Chernoff again and from the negative correlation of the events that different cells are $i$-dense, the probability that there is some $i$ from $1$ to $\log^2 n$ such that the number of $i$-dense cells is at least $i^{-3}\eps n_0$ is bounded by
\begin{align*}
   & \sum_{i=1}^{\lfloor\log ^2 n\rfloor} \binom{n_0}{\lceil i^{-3}\eps n_0\rceil} \exp\rbrac{-\frac{(i n/n_0)^2}{(1 + i/3)n/n_0} \cdot \lceil i^{-3}\eps n_0\rceil}\\
   & \le \sum_{i=1}^{\lfloor\log ^2 n\rfloor} \rbrac{\frac{ei^3}{\eps } \exp\rbrac{- (3/4) i n/n_0}}^{i^{-3}\eps n_0} \\
   & = \sum_{i=1}^{\lfloor\log ^2 n\rfloor} \left[ \rbrac{ \left(\frac{ei^3}{\eps}\right)^{1/i} \exp\rbrac{- (3/4) n/n_0}}^{\eps n_0} \right]^{i^{-2}} \\
   & = \sum_{i=1}^{\lfloor\log ^2 n\rfloor} \left( e^{- \Omega(n)} \right)^{i^{-2}} 
 = e^{- \Omega(n/\log^2n)} = o(1),
\end{align*}
for $\kappa_2$ large enough given $\eps$.
Finally, by one last application of Chernoff, the expected number of cells that are $i$-dense for some $i\ge\log^2 n$ (i.e.~with at least $\lceil \log^2 n\rceil n/n_0$ points) is 
\[
n_0 \exp\rbrac{-\Omega((n/n_0) \log^2 n)}=o(1),
\]
and so a.a.s.\ there are none. 
Summarizing, a.a.s.\ there are at most $\eps n_0$ $0$-dense cells, at most $i^{-3}\eps n_0$ $i$-dense cells for $1\le i\le \log^2 n$ and no $i$-dense cell for $i\ge \log^2n$. Moreover, we can turn any $i$-dense cell into a cell with exactly $N$ points by removing at most $(i+1)n/n_0$ points.
Thus, a.a.s.\ the number of points we need to remove from cells with more than $N$ points is at most 
\[
(n/n_0)\eps n_0 + 
\sum_{i=1}^{\log^2 n} (i+1) (n/n_0) \cdot i^{-3} \eps n_0 \le 4 \eps n
\]
since $\sum_{i=1}^\infty (i+1)i^{-3} \le 3$. This completes the proof of the lemma.
\end{proof}

From now on, we will assume that the a.a.s.\ statements of Lemmas~\ref{lem:edges} and~\ref{lem:trimming} hold and will build a large triangle packing in $\cG_\Td(\bX,r)$ deterministically, although our construction will involve a random experiment.
Let $\bX'$ be the set of points given by Lemma~\ref{lem:trimming}, and let $\cG'$ be the graph with vertex set $\bX'$ where two vertices are adjacent if they are in distinct cells which are adjacent in $\cC$ (in other words, $\cG'$ is the blowup of $\cC$ by factor $N$). The precise location of each point of $\bX'$ is not relevant --- only the cell where it belongs matters.
Note that $\cG'$ is a spanning subgraph of $\cG_\Td(\bX',r)$ by construction. Indeed, any two points in $\bX'$ that are adjacent vertices in $\cG'$ must belong to adjacent cells in $\cC$, and thus must be also adjacent in $\cG_\Td(\bX',r)$. In view of that, any triangle packing in $\cG'$ gives a triangle packing in $\cG_\Td(\bX',r)$.
Furthermore, $\cG'$ is a regular graph (since it is a blowup of the regular graph $\cC$) and, in view of~\eqref{eq:degC} and~\eqref{eq:Nbounds}, every vertex of $\cG'$ has degree at most
\begin{equation}\label{eq:degG}
\theta r^d n_0 N \le (1+\eps/2) \theta r^d n
\end{equation}
Let $\cH'=\cH'(\cG')$ be the 3-uniform hypergraph with vertex set $E(\cG')$ where each edge of $\cH'$ consists of three elements of $E(\cG')$ forming a triangle in $\cG'$.
Note that $3$ points in $\bX'$ determine a triangle in $\cG'$ iff they belong to $3$ different pairwise-adjacent cells in $\cC$.
We will now start to describe a fractional matching $f$ for $\cH'$, which we will use to get a large triangle packing in $\cG'$.

Consider the following random experiment, in which we generate an equilateral triangle in the torus with the appropriate probability distribution. This is one crucial spot in the argument in which we assume $d\ge2$. (For $d=1$ we will consider a different experiment in Section~\ref{sec:d1} in which the triangles obtained are typically far from equilateral.)

\begin{RE}\label{REI}\hspace{0cm}
\begin{enumerate}[(i)]
\item\label{step:REI1}
We pick $3$ random points $Y_1,Y_2,Y_3\in [0,1]^d$ (not necessarily in $\bX$ or $\bX'$) that form an equilateral triangle of side length at most $r$ in the torus. We do this by picking $Y_1$ first uniformly in $[0,1]^d$. Then pick $Y_2$ uniformly in the ball of radius $r$ and center $Y_1$. Finally, pick $Y_3$ uniformly from all possible choices such that $Y_1,Y_2,Y_3$ determine an equilateral triangle in the torus (see Figure~\ref{fig:triangle}). Note that, given $Y_1$ and $Y_2$, point $Y_3$ must lie in a $(d-2)$-dimensional sphere (e.g.~two points when $d=2$ or a circle when $d=3$). Note that the distribution of each of the three points $Y_1, Y_2, Y_3$ individually is uniform on $[0,1]^d$. In fact, our random triangle $Y_1, Y_2, Y_3$ would have the same distribution if, for any permutation $(i, j, k)$ of $(1, 2, 3)$ we generated it by first revealing $Y_i$ (uniform on $[0,1]^d$), then revealing $Y_j$ (uniform on the ball of radius $r$ centered at $Y_i$), and then revealing $Y_k$ (uniform from all points equidistant from $Y_i, Y_j$). So $Y_i, Y_j, Y_k$ has the same distribution as $Y_1, Y_2, Y_3$.

\begin{figure}
\centering
\begin{tikzpicture}[scale=1] 
\tikzstyle{vert}=[shape=circle,draw=black,fill=white, inner sep=.5mm]
\tikzstyle{pnt}=[shape=circle,draw=black,fill=black, inner sep=.5mm]
\usetikzlibrary{decorations.pathreplacing}
\begin{scope}[]
\draw (-1,-1) rectangle (3,3);
\node[pnt,label=left:$Y1$] (Y1) at (1,1){};
\end{scope}
\begin{scope}[xshift=150]
\node[pnt,label=left:$Y1$] (Y1) at (1,1){};
\draw[dashed,fill=gray!15] (Y1) circle (1.2);
\draw (-1,-1) rectangle (3,3);
\node[pnt,label=left:$Y1$] (Y1) at (1,1){};
\draw[dashed] (Y1) circle (1.2);
\node[pnt,label=left:$Y2$] (Y2) at (1.5,1.866){};
\end{scope}
\begin{scope}[xshift=300]
\draw (-1,-1) rectangle (3,3);
\node[pnt,label=left:$Y1$] (Y1) at (1,1){};
\node[pnt,label=above:$Y2$] (Y2) at (1.5,1.866){};
\node[vert,label=left:] (Y4) at (.57, 1.9){};
\node[pnt,label=right:$Y3$] (Y3) at (2, 1){};
\foreach \x/\y in {Y1/Y2,Y2/Y3,Y3/Y1} { \draw (\x) -- (\y); }	
\end{scope}
\end{tikzpicture}
\caption{The random experiment to generate equilateral triangles for $d=2$. After choosing $Y_1$ and $Y_2$ there are exactly two choices for $Y_3$.}
\label{fig:triangle}
\end{figure}
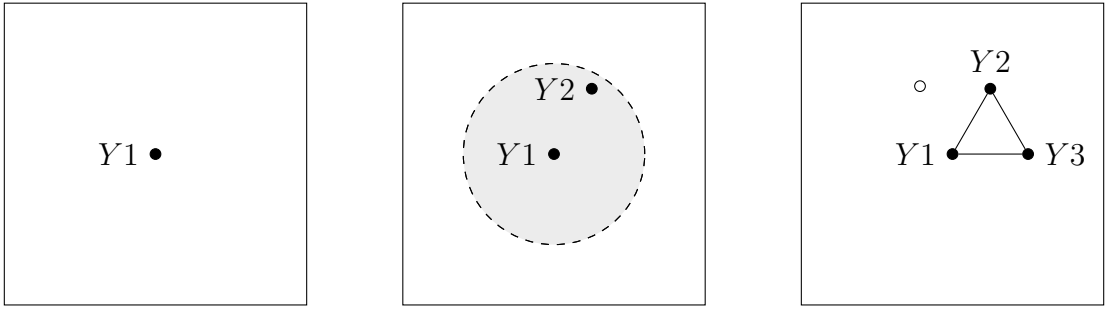

\item\label{step:REI2}
Now we attempt to generate a random triangle $T$ of $\cG'$ (thus an edge of $\cH'$) and a random edge $E$ of $\cG'$ (thus a vertex of $\cH'$) contained in $T$. Moreover, $T$ will be ``nearly'' equilateral in the sense that all three edges will have similar lengths. This attempt may fail (and return $T=\emptyset$ and/or $E=\emptyset$) with some small probability that we will bound later. For each $i=1,2,3$, let $C_i$ be cell that contains the random point $Y_i$. Then we choose one point $X'_i$ from $\bX'\cap C_i$ uniformly at random and independently from all other choices. If $\{C_1,C_2\}$ is an edge of $\cC$ (that is, both cells are different and have centers at distance at most $(1-d\eps)r$), then we define $E=\{X'_1,X'_2\}$. By construction, $E$ is a random edge of the blowup graph $\cG'$. Otherwise, if $C_1,C_2$ are not adjacent in $\cC$, then we {\em fail} to generate an edge of $\cG'$ and set $E=\emptyset$. This can only happen when the distance between $Y_1$ and $Y_2$ is very small or very close to $r$. Similarly, if $\{C_1,C_2,C_3\}$ is a triangle in $\cC$, then we define $T$ to be the triangle in $\cG'$ determined by vertices $X'_1,X'_2,X'_3$. In that case, $E$ is an edge of $T$ (so $E\ne\emptyset$). Moreover, since the joint distribution of $Y_1,Y_2,Y_3$ (and thus of $X'_1,X'_2,X'_3$) is invariant under any permutation of the indices, then $E$ is distributed as choosing the triangle $T$ first and then picking a random edge of $T$ with probability $1/3$. Otherwise, if $\{C_1,C_2,C_3\}$ is not a triangle in $\cC$, then we {\em fail} to generate a triangle in $\cG'$ and simply define $T=\emptyset$. Note that we may have $T=\emptyset$ but $E\ne\emptyset$ if $\{C_1,C_2\}$ is an edge of $\cC$ but $\{C_1,C_2,C_3\}$ is not a triangle of $\cC$. (Again, this may only happen if a distance between pairs of $Y_i$'s is very small or very close to $r$.)
\end{enumerate}
\end{RE}

The following lemma encapsulates some of the key properties of the joint distribution of $T,E$ as described in Random Experiment~\ref{REI}. In Section~\ref{sec:general}, we will describe a more sophisticated generalization of the experiment for which the lemma below will still apply (with a trivial modification of the constants). In Section~\ref{sec:d1}, we will give yet another version of the random experiment (in which the triangles generated are far from equilateral) but for which the resulting joint distribution of $T,E$ will still satisfy the conclusion of the lemma.
\begin{lem}\label{lem:ET}
For any vertex $e$ in $\cH'$ and any edge $t$ in $\cH'$ that contains $e$,
\begin{enumerate}
\item
$\displaystyle \pr(E=e) = \frac{2s^{2d}}{\theta r^d N^2},$
\medskip
\item
$\displaystyle \pr(T=t,E=e) = \frac{1}{3} \pr(T=t)$,
\medskip
\item
$\displaystyle \pr(T=t \mid E=e) = \frac{\theta r^d N^2}{6s^{2d}} \pr(T=t)$.
\end{enumerate}
\end{lem}
\begin{proof}
From our earlier description of the random generation of $T$ and $E$, the distribution of $E$ given that $T=t$ (so in particular $T\ne\emptyset$) is uniform on all three edges of triangle $t$ (i.e.~each edge is chosen with probability $1/3$). This immediately gives part~(2).
To show part~(1), let $e=\{\hat X_1,\hat X_2\}$ be an edge of $\cG'$ with $\hat X_i$ contained in cell $\hat C_i$, for $i=1,2$. The probability that $Y_1,Y_2$ land in cells $\hat C_1,\hat C_2$ (in either order) is $2s^{2d}/\theta r^d$. Indeed, $Y_1$ lands in $C_1$ with probability $s^d$ and, conditional on that, $Y_2$ lands in $\hat C_2$ with probability $s^d/\theta r^d$, since it is uniformly distributed in a ball of radius $r$ around $Y_1$ which fully contains $C_2$.
The factor of $2$ accounts for the probability that $Y_1$ lands in $\hat C_2$ and $Y_2$ lands in $\hat C_1$.
Conditional on $Y_1,Y_2$ being in cells $\hat C_1,\hat C_2$, the probability that we pick precisely points $\hat X_1,\hat X_2$ from these cells is $1/N^2$.
Part~(3) follows immediately from dividing $\pr(T=t,E=e)$ by $\pr(E=e)$.
\end{proof}

In view of Lemma~\ref{lem:ET}, $\pr(T=t\mid E=e)$ does not depend on our choice of $e$ (as long as $e\subset t$). Hence, for our fractional matching $f$ on $\cH'$, if $t$ is a triangle of $\cG'$ we will let
\begin{equation}
\label{eq:fdef}
f(t) = \pr(T=t\mid E=e) = \frac{\theta r^d N^2}{6s^{2d}} \pr(T=t),
\end{equation}
where $e$ is any edge of $\cG'$ contained in $t$. To check that $f$ is indeed a fractional matching, note that for any edge $e$ of $\cG'$,
\begin{equation}
\label{eq:fracmatch}
\suma{t\in \cH'\\t\supset e} f(t) = \suma{t\in \cH'\\t\supset e} \pr(T=t\mid E=e) = \pr(T\ne\emptyset \mid E=e) \le 1,
\end{equation}
where the second equality follows from the law of total probability, since (conditional on $E=e$) $T$ must either be a triangle containing $e$ or $T=\emptyset$ if our experiment fails to produce a triangle.

In order to bound $\alpha(f)$, we first observe that, for any triangle $t$ and any edge $e$ contained in that triangle,
\begin{equation}\label{eq:ftbound}
f(t) = \pr(T=t\mid E=e) \le 1/N,
\end{equation}
since (conditional on $E=e$) we may take for granted that the third point of $T$ lands in the correct cell but still need to pick the correct point of $\bX'$ in that cell. From~\eqref{eq:ftbound} and since any two vertices of $\cH'$ (i.e.~edges of $\cG'$) are in at most one edge of $\cH'$ (i.e.~a triangle),  we have 
\begin{equation}\label{eq:alphabound}
\alpha(f) \le \frac 1N  \le \frac{n_0}{n} \le \left(\frac{2}{\eps\kappa_2}\right)^d, 
\end{equation}
where we have used the definition of $N$ and the upper bound of $n_0$ in~\eqref{eq:n0range}. Let $\sigma$ be the number whose existence is guaranteed by Lemma~\ref{lem:Kahn} for $\ell=3$ and our value $\eps$. Then we can ensure that $\alpha(f) < \sigma$ by choosing $\kappa_2 > 2 \sigma^{-1/d} \eps^{-1}$.
Now Lemma~\ref{lem:Kahn} tells us that $\cH'$ has a matching of size at least $(1-\eps)f(\cH')$ which we will estimate below.

\begin{lem}\label{lem:fH}
Let $c=3d^2+1$ and suppose $\epsilon>0$ is small enough given $d$. Then
\[
f(\cH') \ge (1-c\eps) \frac{1}{6}\theta r^d n^2.
\]
\end{lem}
The proof of Lemma~\ref{lem:fH} can be easily extended to cover the case of general $F$-packings for $F\in\cU(d)$ (with an easy modification of some of the constants). We will discuss that in Section~\ref{sec:general} in more detail (cf.~Lemma~\ref{lem:fH2}). The statement of the lemma is also true for the triangle case in dimension $d=1$. However, the proof for that case will be somewhat different than the one here (since the triangles will no longer be equilateral) and will be argued in Section~\ref{sec:d1}.
\begin{proof}[Proof of Lemma~\ref{lem:fH}]
The conclusion will follow easily from the two claims below.
\begin{claim}\label{claim:x}
Suppose an edge $e$ of $\cG'$ has endpoints in the cells $\hat C_1, \hat C_2$. Let the distance between the centers of $\hat C_1$ and $\hat C_2$ be $x$, and suppose 
  \begin{equation}\label{eq:xbounds}
      3d \eps r \le x \le (1-3d\eps)r.
  \end{equation}
  Then  $\pr(T\ne\emptyset \mid E = e)  = 1$.
\end{claim}
\begin{claim}\label{claim:A}
Let $A$ be the set of edges in $\cG'$ that satisfy the hypothesis of Claim~\ref{claim:x}. Then
\[
|A| \ge \frac{1}{2}(1-8d^2\eps)\theta r^d n^2.
\]
\end{claim}
Assuming the two claims to be true, we conclude that
\begin{equation}
\label{eq:fHbound}
f(\cH') = \sum_{t\in\cH'} f(t) = \frac{1}{3} \sum_{e \in V(\cH')} \suma{t\in\cH'\\ t\supset e} \pr(T=t\mid E=e)
= \frac{1}{3}  \sum_{e \in V(\cH')} \pr(T\ne\emptyset\mid E=e) \ge |A|/3,
\end{equation}
where the $1/3$ factor accounts for the fact that triangles in the double sum are counted $3$ times each, and the last inequality follows from Claim~\ref{claim:x} and restricting the last sum to only $e\in A$. The proof of the lemma follows immediately from~\eqref{eq:fHbound} and Claim~\ref{claim:A}.

It just remains to prove Claims~\ref{claim:x} and~\ref{claim:A}. For the first claim, let $Y_1,Y_2,Y_3$ be the random points in Random Experiment~\ref{REI} (which form an equilateral triangle in the torus $\Td$ of edge length $R$) and let $C_1,C_2,C_3$ be the random cells that contain them, respectively. Conditional on $E=e$, we must have $\{C_1,C_2\}=\{\hat C_1,\hat C_2\}$ (w.l.o.g.~assume $C_1= \hat C_1$ and $C_2= \hat C_2$). In particular and since cells have side $s\le \eps r$, we must have
\[
x - d\eps r \le R \le x + d\eps r.
\]
Then, for $i=1,2$, the distance between the centers of $C_i$ and $C_3$
is at least
\[
R - d\eps r \ge x - 2d\eps r \ge d\eps r
\]
and at most
\[
R + d\eps r \le x + 2d\eps r \le (1-d\eps)r.
\]
Hence, $C_3$ is different from and adjacent to both of $C_1,C_2$ in the graph of cells $\cC$. As a result, the experiment always succeeds at generating a triangle of $\cG'$, i.e.~$T\ne\emptyset$. This finishes the proof of Claim~\ref{claim:x}.

Instead of proving Claim~\ref{claim:A}, we will prove a slightly more general statement which we will also use later on in Section~\ref{sec:general}.
\begin{claim}\label{claim:Ageneral}
For $\alpha_1,\alpha_2>0$, let $A(\alpha_1,\alpha_2)$ be the set of all edges $e$ of $\cG'$ such that the distance $x$ between the centers of the cells that contain the endpoints of $e$ satisfies
\[
\alpha_1 d\eps r \le x \le (1-\alpha_2 d\eps) r.
\]
Then $\displaystyle |A(\alpha_1,\alpha_2)| \ge \frac12 \left( (1-\alpha_2 d\eps)^d - (\alpha_1 d\eps)^d \right) \theta r^d n^2$.
\end{claim}

Clearly, setting $\alpha_1=\alpha_2=3$ in Claim~\ref{claim:Ageneral}, gives
\[
|A| = |A(3,3)| \ge \frac12 \left( (1-3 d\eps)^d - (3 d\eps)^d \right) \theta r^d n^2
\ge \frac12 \left( 1 - 3 d^2\eps - (3 d\eps)^d \right) \theta r^d n^2.
\]
Then, assuming $(3 d\eps)^d < \eps$ (which is possible for small $\eps$ since $d\ge2$), we obtain Claim~\ref{claim:A}.

Finally, we proceed to prove Claim~\ref{claim:Ageneral}.
Assume $1-\alpha_2 d\eps>0$ (otherwise $|A(\alpha_1,\alpha_2)|$ is trivially $0$).
Let $e=\{X'_i,X'_j\}$ be a generic edge in $A(\alpha_1,\alpha_2)$. Let $C_i,C_j$ be the cells containing points $X'_i,X'_j$, respectively.
There are $n_0$ ways to pick $C_i$. Given $C_i$, let $L$ be the set of points in the torus $\Td$ that are at distance at least $\alpha_1 d\eps r$ and at most $(1-\alpha_2 d\eps) r$ from the center of $C_i$. Clearly, the center of $C_j$ (if such an edge $e$ is possible) must belong to $L$, so $C_j$ intersects $L$. There are at least
\[
\vol(L)/s^d = \vol(L) n_0 = \left( (1-\alpha_2 d\eps)^d - (\alpha_1 d\eps)^d \right) \theta r^d n_0
\]
cells intersecting $L$ and thus choices for $C_j$. Therefore, we have at least
\[
\frac12  \left( (1-\alpha_2 d\eps)^d - (\alpha_1 d\eps)^d \right) \theta r^d {n_0}^2
\]
choices for $\{C_i,C_j\}$, where the $1/2$ factor comes from the fact that we are counting unordered pairs of cells. We now must choose one point in $\bX'\cap C_i$ for $i=1,2$. This can be done in $N^2\ge (n/n_0)^2$ ways. Putting it all together,
\[
|A(\alpha_2,\alpha_2)| \ge \frac12 \left( (1-\alpha_2 d\eps)^d - (\alpha_1 d\eps)^d \right) \theta r^d n^2,
\]
which proves Claim~\ref{claim:Ageneral}.
\end{proof}

As a conclusion, we have found a matching in $\cH'$ of size at least
\begin{equation}
\label{eq:almostdone}
(1-\eps)f(\cH') \ge (1-(c+1)\eps) \theta r^d n^2/6.
\end{equation}
This matching in $\cH'$ corresponds to a triangle packing in $\cG'$ and thus in $\cG_{\Td}(\bX',r)$, but we would like a triangle packing in $\cG_{\Td}(\bX,r)$. We simply remove the points in $\bX' \setminus \bX$,  deleting any triangles from our packing that use those points. Since $|\bX' \setminus \bX|  \le 4\eps n $ by Lemma~\ref{lem:trimming} and in view of~\eqref{eq:degG}, we lose at most
\begin{equation}
\label{eq:edges_killed}
4\eps (1+\eps/2)\theta r^d n^2 \le 36\eps \theta r^d n^2 / 6
\end{equation}
edges and thus at most the same amount of triangles (where we also assumed $\eps<1$).
Then we add the points in $\bX \setminus \bX'$, which does not decrease our triangle packing.
So in view of Lemma~\ref{lem:edges} and choosing $\eps$ to be sufficiently small given $d$ and $\delta$,
we get a packing with at least $(1-\delta) m(\cG_\Td(\bX,r)) / 3$ triangles. This proves Theorem~\ref{thm:dense} for the case $d\ge 2$.
\end{proof}

\section{General $F$. Proof of Theorem~\ref{thm:unitdistance}}\label{sec:general}

\begin{proof}
This proof is a generalization of that of Theorem~\ref{thm:dense} (for $d\ge 2$) in Section~\ref{sec:triangles}. As a result, we will trace the argument in that section at a rather brisk pace and only elaborate at the specific points that require further detail.

Fix some $F \in \cU(d)$ for $d\ge1$. Assume $F$ has no isolated vertices (if $F$ has isolated vertices then we can pack the subgraph induced by non-isolated vertices). Since $F \in \cU(d)$, we have $F \subseteq F'$ where $F'$ is the blowup of a distance graph in $\real^d$. We can assume $V(F)=V(F')$ by taking $F'$ to be the blowup where we duplicate each vertex the smallest possible number of times.
We label the vertices of the distance graph $F''$ as $v_1,\ldots,v_{n_{F''}}$, where $n_{F''}$ denotes the number of vertices in $F''$. For each $1\le i\le n_{F''}$, let $v_{i,1},\ldots,v_{i,b_i}$ be the vertices in the blowup graph $F'$ (or in $F$) that correspond to vertex $v_i$ in $F''$. We call $b_i$ the {\em multiplicity} of $v_i$ in the blowup $F'$ (or in $F$). In other words, $b_i$ is the number of times $v_i$ is duplicated in the blowup. Define $b=\max\{b_i\mid 1\le i\le n_{F''}\}$ to be the largest multiplicity.
Consider a fixed arbitrary drawing $D$ of $F''$ in  $\mathbb{R}^d$ where every edge has Euclidean distance 1.
For each $1\le i\le n_{F''}$, let $Z_i\in\real^d$ denote the position of vertex $v_i$ in the drawing $D$.

Recall that an $F$-packing in a graph $G$ is a collection of edge-disjoint copies of $F$ in $G$.
In view of Lemma~\ref{lem:edges} and Remark~\ref{rem:torus} (but replacing triangles by copies of $F$), any $F$-packing in the toroidal model $\cG_\Td(\bX,r)$ of size at least $(1-\delta/2) m(\cG_{\Td}(\bX,r))/m_F$ can be turned into one in the cubic model $\cG_\Qd(\bX,r)$ of size at least 
\[
(1-\delta/2 - (2m_Fd+o(1))r )m(\cG_{\Td}(\bX,r))/m_F \ge (1-\delta)m(\cG_{\Qd}(\bX,r))/m_F
\]
if we pick $\kappa_3=\kappa_3(F,\delta,d)$ to be sufficiently small. As a result, it will suffice to prove Theorem~\ref{thm:unitdistance} for the toroidal case, that is $*=\Td$.

As in Section~\ref{sec:triangles}, we partition the torus into $n_0$ cubic cells of side length $s=\lceil 1/\eps r\rceil^{-1} \le  \eps r$ for some $\eps>0$ to be chosen sufficiently small later. $\cC$ is the graph of cells again, where two cells are adjacent if their centers have distance at most $(1-d\eps)r$. Recall that two points in adjacent cells must be within distance $r$.
We apply  Lemma~\ref{lem:trimming} to obtain our new set of points $\bX'$ with $N=\lceil n/n_0\rceil$ points in each cell by slightly modifying $\bX$. From~\eqref{eq:n0range} and~\eqref{eq:Nbounds} and for sufficiently large $\kappa_2=\kappa_2(F,d,\delta)$, we have
\begin{equation}\label{eq:Nb}
N \ge n/n_0 \ge (\eps \kappa_2/2)^d > b,
\end{equation}
where recall $b$ is the maximum multiplicity of all vertices of $F''$ in the blowup $F'$. Let $\cG'$ be the graph with vertex set $\bX'$ where two vertices are adjacent if they are in distinct cells which are adjacent in $\cC$. (In other words, $\cG'$ is the blowup of $\cC$ by factor $N$.) 
As in Section~\ref{sec:triangles}, let $\cH'= \cH'(\cG')$ be the hypergraph of (edge sets of) copies of $F$ in $\cG'$.
(That is, the vertices of $\cH'$ are the edges of $\cG'$ and the edges of $\cH'$ are the copies of $F$ in  in $\cG'$.)

We will assume that the a.a.s.\ conclusions of Lemmas~\ref{lem:edges} and~\ref{lem:trimming} hold, and build a large fractional matching $f$ for $\cH'$ deterministically.
To do so, we will consider a random procedure that attempts to generate a pair $(T,E)$, where $T$ is a copy of $F$ in $\cG'$ (thus an edge of $\cH'$) and $E$ an edge of $\cG'$ (thus a vertex of $\cH'$) contained in that copy.
This is a generalization of Random Procedure~I described in Section~\ref{sec:triangles}.
Technically, this experiment can fail, in which case it does not produce the desired copy of $F$ and/or edge.

\begin{RE}\label{REII}\hspace{0cm}
  \begin{enumerate}[(i)]
      \item\label{step:Dprime} We choose a random $R \in [0, r]$ distributed as the distance between two independent uniformly distributed points in the torus $\Td$, conditional on the distance being at most $r$. Equivalently, $R$ is distributed
as the distance from the origin in $\real^d$ to a uniform random point in the closed ball of radius $r$ centered at the origin. We shrink our drawing $D$ by a factor of $R$ to obtain a new drawing $D' \subseteq \mathbb{R}^d$ of the distance graph $F''$, in which all edges have Euclidean distance $R$.
In other words, for $1\le i\le n_{F''}$, the position $Z_i$ of vertex $v_i$ in our original drawing $D$ becomes $Z'_i=RZ_i$ in the new drawing $D'$.

      \item\label{step:Dpprime} We take a uniform random rotation of $D'$ to get a new drawing $D''$.
More precisely, we pick a random $d\times d$ orthogonal matrix $Q$, where each column $j=1,\ldots,d$ is chosen uniformly at random among all unit vectors that are orthogonal to the previous $j-1$ columns (so it is chosen from a $(d-j)$-dimensional sphere uniformly at random). Then, the new position of each vertex $v_i$ of $F''$ becomes $Z''_i=QZ'_i$ in drawing $D''$.

      \item\label{step:Dppprime} We take a uniform random translation mapping $D''$ into the torus $\Td$. More precisely, we pick a random uniform point $Z\in[0,1]^d$, translate each vertex of $F''$ in the drawing $D''$ by $Z$ and reduce all coordinates modulo $1$. Hence, we obtain a new drawing $D'''$ of $F''$ in the torus where each vertex $v_i$ in $F''$ ($1\le i\le n_{F''}$) is located at the point $Y_i=Z''_i+Z \quad (\mathrm{mod}\; 1)$ in $\Td$. For convenience, we define $Y_{i,j}=Y_j$ for all $j=1,\ldots,b_i$, so each vertex $v_{i,j}$ in $F$ (or in $F'$) has an associated location $Y_{i,j}$ in the torus as well.

\item\label{step:Xprime}
For each $1\le i\le n_{F''}$, let $C_i$ be the cell containing the random point $Y_i$. We choose a set of $b_i$ different points $X'_{i,1},\ldots,X'_{i,b_i}$ in $\bX'\cap C_i$ uniformly at random and independently from other choices. This is always possible, since $b_i\le b<N$ by~\eqref{eq:Nb}. Hence, to each vertex $v_{i,j}$ of $F$ we have associated a random point $X'_{i,j}$ in $\bX'\cap C_i$.

\item\label{step:E} We pick an edge $\{v_{i_1,j_1},v_{i_2,j_2}\}$ of $F$ uniformly at random and independently of all other choices. Recall $C_{i_1},C_{i_2}$ are the cells containing points $Y_{i_1,j_1},Y_{i_2,j_2}$ (or equivalently points $Y_{i_1},Y_{i_2}$), respectively. If $\{C_{i_1},C_{i_2}\}$ is an edge of the graph of cells $\cC$ (i.e.~the two cells are different and have their centers at distance at most $(1-d\eps)r$), then we output $E=\{X'_{i_1,j_1},X'_{i_2,j_2}\}$, with $X'_{i,j}$ defined in the previous step. By construction, $E$ is a random edge of $\cG'$ (and thus a vertex of $\cH'$). Otherwise, if $\{C_{i_1},C_{i_2}\}$ is not an edge of $\cC$, then we set $E=\emptyset$, and declare that our experiment {\em fails} to produce an edge of $\cG'$. This can only happen if $R$ is very small or very large (close to $r$).

\item\label{step:T} Now recall that $C_1,\ldots,C_{n_{F''}}$ are the cells containing points $Y_1,\ldots,Y_{n_{F''}}$.
Suppose first that $C_1,\ldots,C_{n_{F''}}$ (regarded as vertices in the graph of cells $\cC$) form a labelled copy of the distance graph $F''$ in $\cC$. That is, they are all different cells and, for each edge $\{v_i,v_j\}$ in $F''$, the centers of the corresponding cells $C_i,C_j$ are at distance at most $(1-d\eps)r$. This is a stronger condition than that in step~\eqref{step:E}, so in particular $E\ne\emptyset$.
In that case, the points $X'_{i,j}$ ($1\le i\le n_{F''}$, $1\le j\le b_i$) determine a labeled copy $T$ of $F$ in $\cG'$ (and thus an edge of $\cH'$). We output such $T$. Moreover, by construction, $E$ is an edge of $T$ chosen uniformly at random among those.
Otherwise, in the case that $C_1,\ldots,C_{n_{F''}}$ do not form a labelled copy of $F''$ in $\cC$, we simply set $T=\emptyset$ and declare that our experiment {\em fails} to produce a copy of $F$ in $\cG'$. Note that we may have $T=\emptyset$ but $E\ne\emptyset$ (if $\{C_{i_1},C_{i_2}\}$ is an edge of $\cC$ but $C_1,\ldots,C_{n_{F''}}$ do not form a labeled copy of $F''$ in $\cC$.)  
  \end{enumerate}
\end{RE}

\begin{remark}\label{rem:nowrap}
Our constant $\kappa_3$ will be chosen small enough that the distance between vertices in $D'''$ is unaffected when we reduce each coordinate modulo 1 to map it from $\mathbb{R}^d$ onto $\Td$. If we choose $\kappa_3 \le 1/2 \diam(D)$ then any pair of vertices in $D'$ have distance at most $1/2$, which suffices. In other words, our drawing $D'''$ of $F''$ does not wrap around the torus.
\end{remark}

Observe that the above procedure is a natural generalization of Random Experiment~\ref{REI} from Section~\ref{sec:triangles}. Indeed, for $F=F''=K_3$ (i.e.~a triangle) and $d\ge 2$, the distribution of $Y_1,Y_2,Y_3$ obtained in step~\eqref{step:Dppprime} coincides with that in step~\eqref{step:REI1} of Random Experiment~\ref{REI}. The only minor caveat is that here in step~\eqref{step:E} we randomize the choice of edge $\{v_{i_1,j_1},v_{i_2,j_2}\}$ whereas in Section~\ref{sec:triangles} we did not have to, due to the symmetric role of all edges in an equilateral triangle.
In fact, many of the key properties of Random Experiment~\ref{REI} naturally extend to Random Experiment~\ref{REII}.
\begin{enumerate}
\item
For any copy $t$ of $F$ in $\cG'$, conditional on $T=t$, the distribution of $E$ is uniform among all edges of $t$ (see step~\eqref{step:T}).
\item 
The location $Y_{i,j}$ of each vertex $v_{i,j}$ of $F$ (for $1\le i\le n_{F''}$ and $1\le j\le b_i$) is uniformly distributed on the torus. This follows from the construction of $D'''$ in step~\eqref{step:Dppprime} as a uniform random translation of $D''$.
\item
For any edge $\{v_{i_1,j_1},v_{i_2,j_2}\}$ of $F$ and conditional on $Y_{i_1,j_1}$, the distribution of $Y_{i_2,j_2}$ is uniform in a ball of radius $r$ centered at $Y_{i_1,j_1}$. This is still true if $\{v_{i_1,j_1},v_{i_2,j_2}\}$ is a random edge as in step~\eqref{step:E}. This follows from the distribution of the scaling factor $R$ in step~\eqref{step:Dprime} and the uniform random rotation in step~\eqref{step:Dpprime}.
\end{enumerate}
In view of these properties, the joint distribution of $T,E$ from Random Experiment~\ref{REII} satisfies this natural extension of Lemma~\ref{lem:ET}.

\begin{lem}\label{lem:ET2}
For any vertex $e$ in $\cH'$ and any edge $t$ in $\cH'$ that contains $e$,
\begin{enumerate}
\item
$\displaystyle \pr(E=e) = \frac{2s^{2d}}{\theta r^d N^2},$
\medskip
\item
$\displaystyle \pr(T=t,E=e) = \frac{1}{m_F} \pr(T=t)$,
\medskip
\item
$\displaystyle \pr(T=t \mid E=e) = \frac{\theta r^d N^2}{2m_Fs^{2d}} \pr(T=t)$.
\end{enumerate}
\end{lem}
The proof follows nearly verbatim from that of~Lemma~\ref{lem:ET} (just replace triangle by copy of $F$, $3$ by $m_F$ and $Y_1,Y_2$ by $Y_{i_1,j_1},Y_{i_2,j_2}$) so we omit it.
As a result, we can define for each copy $t$ of $F$ in $\cG'$,
\[
f(t) = \pr(T=t\mid E=e) = \frac{\theta r^d N^2}{2m_Fs^{2d}} \pr(T=t),
\]
exactly as we did in~\eqref{eq:fdef}, where $e$ is an arbitrary edge contained in $t$. By~\eqref{eq:fracmatch}, $f$ is a fractional matching on $\cH'$.

Now we bound $\alpha(f)$. Given two different edges $e,e'$ of $\cG'$, let $v$ be a vertex of $\cG'$ that is in $e'$ but not in $e$ (there must be at least one of these). Then
\begin{equation}
\label{eq:PvinT}
\suma{t\in\cH'\\t\supset e,e'} f(t) = \suma{t\in\cH'\\t\supset e,e'} \pr(T=t\mid E=e) \le \pr(T\ni v\mid E=e).
\end{equation}
To find an upper bound of $\pr(T\ni v \mid E=e)$, we take for granted that all the random points $Y_{i,j}$ of Random Experiment~\ref{REII} land in the right cells, which includes the cell $C$ containing $v$. The probability that vertex $v$ is picked from $\bX'\cap C$ in step~\eqref{step:Xprime} of the experiment is at most $b/N$, where recall, where $b$ is the maximum of all $b_i$. In view of~\eqref{eq:PvinT}, we conclude that
\[
\alpha(f) = \max_{\substack{e,e'\in V(\cH')\\e\ne e'}} \suma{t\in\cH'\\t\supset e,e'} f(t) \le b/N.
\]
We would like to make $b/N< \sigma$, where $\sigma$ is the value guaranteed by Lemma~\ref{lem:Kahn} for $\ell=m_F$ and our value of $\eps$.
From~\eqref{eq:Nb}, $N\ge (\eps\kappa_2/2)^d$. So, assuming that $\kappa_2=\kappa_2(F,d,\delta)$ is sufficiently large, we can achieve $\alpha(f) \le b/N < \sigma$.
Hence Lemma~\ref{lem:Kahn} tells us that $\cH'$ has a matching of size at least $(1-\eps)f(\cH')$.

The following natural extension of Lemma~\ref{lem:fH}, which we will prove at the end of this section, gives a lower bound on $f(\cH')$.
\begin{lem}\label{lem:fH2}
There exists $c=c(F,d)>0$ such that, for any $\eps>0$ sufficiently small given $d$,
\[
f(\cH') \ge (1-c\eps) \frac{1}{2m_F}\theta r^d n^2.
\]
\end{lem}
Hence (assuming Lemma~\ref{lem:fH2}), we have shown the existence of a matching in $\cH'$ of size at least
\[
(1-\eps)f(\cH') \ge (1-(c+1)\eps) \theta r^d n^2/2m_F,
\]
which corresponds to an $F$-packing in $\cG' \subseteq \cG_\Td(\bX',r)$. In order to turn it into an $F$-packing in $\cG_\Td(\bX,r)$, we follow the same argument at the end of Section~\ref{sec:triangles} nearly verbatim (replacing triangles by copies of $F$ and updating constants appropriately), so we just sketch it here. We remove the points in $\bX' \setminus \bX$, which deletes at most (cf.~\eqref{eq:edges_killed})
\[
4\eps (1+\eps/2)\theta r^d n^2 \le 6m_F\eps \theta r^d n^2 / m_F
\]
edges and (thus copies of $F$ from the packing), and then add the points in $\bX \setminus \bX'$, which does not decrease the $F$-packing. 
From Lemma~\ref{lem:edges} and choosing $\eps$ to be sufficiently small given $F$, $d$ and $\delta$,
we get a packing with at least $(1-\delta) m(\cG_\Td(\bX,r)) / m_F$ copies of $F$.
This proves Theorem~\ref{thm:unitdistance}.
\end{proof}

It only remains to prove Lemma~\ref{lem:fH2}. The argument is an easy extension of that of Lemma~\ref{lem:fH}, so we will just sketch the differences.
\begin{proof}[Proof of Lemma~\ref{lem:fH2}]
Let $\beta$ be the smallest Euclidean distance between different vertices in our initial drawing $D$ of $F''$ in $\real^d$. Note $0<\beta\le 1$, since different vertices must occupy different positions in $\real^d$ and adjacent vertices are at distance $1$.
Distances get scaled by a random factor of $R$ in $D'$ on step~\eqref{step:Dprime} of Random Experiment~\ref{REII}. They do not change on steps~\eqref{step:Dpprime} and~\eqref{step:Dppprime} even when projected to the torus, in view of Remark~\ref{rem:nowrap}.
As a result, the smallest toroidal distance between different vertices in the drawing $D'''$ of $F''$ is
\begin{equation}
\label{eq:lambdaR}
\min \{ \dist_\Td(Y_i,Y_j) : 1\le i,j\le n_{F''}, i\ne j \} = \beta R.
\end{equation}

Next, we introduce a claim that extends Claim~\ref{claim:x} in the proof of Lemma~\ref{lem:fH} to the new Random Experiment~\ref{REII} setting.
\newcounter{temp}
\setcounter{temp}{\value{claim}}
\setcounter{claim}{0}
\renewcommand{\theclaim}{\arabic{claim}$'$}%
\begin{claim}\label{claim:x2}
Suppose an edge $e$ of $\cG'$ has endpoints in the cells $\hat C_{i_1}, \hat C_{i_2}$. Let the distance between the centers of $\hat C_1$ and $\hat C_2$ be $x$, and suppose 
  \begin{equation}\label{eq:xbounds}
      (2/\beta+1)d \eps r \le x \le (1-3d\eps)r.
  \end{equation}
  Then  $\pr(T\ne\emptyset \mid E = e)  = 1$.
\end{claim}
\setcounter{claim}{\value{temp}}
\renewcommand{\theclaim}{\arabic{claim}}

Now let $A=A(2/\beta+1, 3)$ be the set of edges of $\cG'$ which satisfy the assumptions of Claim~\ref{claim:x2}. By Claim~\ref{claim:Ageneral} in the proof of Lemma~\ref{lem:fH} (with $\alpha_1=2/\beta + 1, \alpha_2=3$), we get
\[
|A| \ge \frac12 \left( (1-(2/\beta + 1) d\eps)^d - (3 d\eps)^d \right) \theta r^d n^2.
\]
Setting $c=(2/\beta + 1) d^2 + 3 d$, which depends only on $d$ and $D$ (and thus on $F$), and assuming $\eps<1/3d$, we conclude
\begin{equation}
\label{eq:Ageneral}
|A| \ge \frac12 \left( (1-(2/\beta + 1) d^2\eps) - 3 d\eps \right) \theta r^d n^2 = \frac12 (1-c\eps) \theta r^d n^2,
\end{equation}
which plays a role akin to that of Claim~\ref{claim:A} in the proof of Lemma~\ref{lem:fH}.

In view of all the above, we can immediately derive an analogue to~\eqref{eq:fHbound} in this more general setting.
\begin{equation}
\label{eq:fHbound2}
f(\cH') = \sum_{t\in\cH'} f(t) = \frac{1}{m_F} \sum_{e \in V(\cH')} \suma{t\in\cH'\\ t\supset e} \pr(T=t\mid E=e)
= \frac{1}{m_F}  \sum_{e \in V(\cH')} \pr(T\ne\emptyset\mid E=e) \ge \frac{|A|}{m_F},
\end{equation}
where the $1/m_F$ factor accounts for the fact that copies of $F$ in the double sum are counted $m_F$ times each, and the last inequality follows from Claim~\ref{claim:x2} and restricting the last sum to only $e\in A$. The proof of the lemma follows immediately from~\eqref{eq:fHbound2} and~\eqref{eq:Ageneral}.

Finally, we proceed to prove Claim~\ref{claim:x2}.
Let $R$ be distributed as in step~\eqref{step:Dprime} of Random Experiment~\ref{REII}, and let $Y_i$ ($1\le i\le n_{F''}$) be the random points in $\Td$ generated in step~\eqref{step:Dppprime}.
Condition on $E=e$ for some edge $e$ of $\cG'$ satisfying the assumptions of the claim.
In particular, and using the fact that cells have side $s\le\eps r$, we must have
\[
x-d\eps r\le R \le x+d\eps r.
\]
For any pair of different vertices $v_i,v_j$ of $F''$, the centers of the cells $C_i,C_j$ containing $Y_i,Y_j$ are at distance at least
\[
\beta R-d\eps r \ge \beta (x-d\eps r)-d\eps r \ge  d\eps r,
\]
so cells $C_i,C_j$ must be different. Moreover, if $\{v_i,v_j\}$ is an edge of $F''$, the centers of the cells $C_i,C_j$ are at distance at most
\[
R+ds \le x+2ds \le x+2d\eps r \le (1-d\eps)r,
\]
so cells $C_i,C_j$ are adjacent. Hence step~\eqref{step:T} of Random Experiment~\ref{REII} does not fail and $T\ne\emptyset$. This finishes the proof of Claim~\ref{claim:x2}.
\end{proof}

\section{Proof of Theorem~\ref{thm:dense} (for $d=1$)}\label{sec:d1}
\begin{proof}
Most of the proof of Theorem~\ref{thm:dense} in Section~\ref{sec:triangles} also works unchanged for dimension $d=1$, so we just explain in detail the parts of the proof that need to be modified.

We adopt the same definitions and notation from that section, and follow the same construction of the graph of cells $\cC$. We also assume that the a.a.s.\ conclusions of Lemmas~\ref{lem:edges} and~\ref{lem:trimming} hold. This allows us to define $\bX'$, $\cG'$ and $\cH'$ as in Section~\ref{sec:triangles}. In view of Remark~\ref{rem:torus}, it suffices to prove the result for the random geometric graph $\cG_{\Tone}(\bX,r)$ on the torus, since the cubic case can be reduced to it. Note that the one-dimensional torus $\Tone$ is simply a circle of unit perimeter and the cells are intervals of length $s\le\eps r$, where $\eps>0$ will be chosen to be sufficiently small given $\delta$. As in the previous sections, we will sometimes use $\Tone$ and $[0,1]$ interchangeably by identifying points $0$ and $1$ as a single point in the torus (circle).

The main ingredient that needs to be modified is Random Experiment~\ref{REI}, since it generates equilateral triangles in $[0,1]^d$, and this is not possible in dimension $d=1$. We will replace this procedure by the following one, in which triangles have two shorter and one longer edge, but the average distribution of their lengths is uniform.

\begin{RE}\label{REIII}\hspace{0cm}
\begin{enumerate}[(i)]
\item\label{step:REIII1} Let $R_3=\max\{U_1,U_2,U_3\}$, where $U_1,U_2,U_3$ are independent uniform random variables on $[0,r]$. Let $R_1$ be uniformly distributed on $[0,R_3]$ and let $R_2=R_3-R_1$ (so $R_1$ and $R_2$ have the same distribution). Let $R$ be a random variable that, for $i=1,2,3$ takes value $R_i$ with probability $1/3$. We will show (see Lemma~\ref{lem:uniform} at the end of this section) that $R$ is uniformly distributed on $[0,r]$. So remarkably, while none of $R_1,R_2,R_3$ is uniform on $[0,r]$, choosing one of them at random is.
\item\label{step:REIII2}
Let $W_1$ be a random uniform point in the torus $\Tone$ (chosen independently from $R_1,R_2,R_3$). Define $W_2=W_1+R_1$ and $W_3=W_2+R_2=W_1+R_3$, all mod~$1$ since they are in $\Tone$ (see Figure~\ref{fig:dim1}). Now, for $i=1,2,3$, let $Y_i=W_{\sigma(i)}$, where $\sigma$ is a random permutation of the indices $1,2,3$. By construction, each vertex $Y_i$ of the triangle $\{Y_1,Y_2,Y_3\}$ is uniformly distributed on the torus. Moreover, each edge $\{Y_i,Y_j\}$ ($i\ne j$) has length distributed as $R$ and thus uniform on $[0,r]$.
\begin{figure}
\begin{tikzpicture}[scale=1.2]

\draw[thick] (0,0) ellipse (2.8 and 1.4);

\node[right] at (2.8,0.8) {$\Tone$};

\fill (-1.3,-1.25) circle (2pt);
\fill (0,-1.4) circle (2pt);
\fill (1.8,-1.08) circle (2pt);

\node[below] at (-1.3,-1.25) {$W_1$};
\node[below] at (0,-1.4) {$W_2$};
\node[below] at (1.85,-1.08) {$W_3$};

\draw[decorate,decoration={brace,mirror,amplitude=4pt}]
(-1.15,-1.35) -- (-0.15,-1.45)
node[midway,below=3pt] {$R_1$};

\draw[decorate,decoration={brace,mirror,amplitude=4pt}]
(0.15,-1.48) -- (1.6,-1.23)
node[midway,below=3pt] {$R_2$};

\draw[decorate,decoration={brace,amplitude=5pt}]
(-1.3,-1.1) -- (1.8,-0.95)
node[midway,above=3pt] {$R_3$};

\end{tikzpicture}
\caption{Random triangle $\{W_1,W_2,W_3\}=\{Y_1,Y_2,Y_3\}$ in $\Tone$ as on step~\eqref{step:REIII2} of Random Experiment~\ref{REIII}. }\label{fig:dim1}
\end{figure}
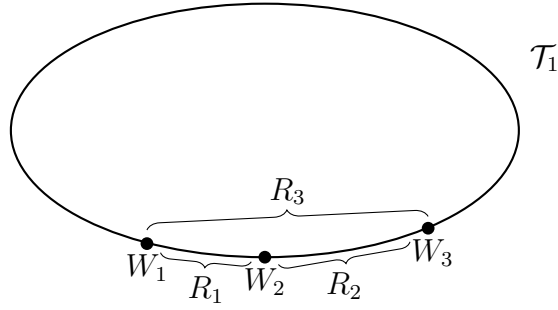

\item\label{step:REIII3}
Generate a random triangle $T$ of $\cG'$ (thus an edge of $\cH'$) and a random edge $E$ of $\cG'$ (thus a vertex of $\cH'$) contained in $T$, exactly as on step~\eqref{step:REI2} of Random Experiment~\ref{REI}. More precisely, $E=\{X'_1,X'_2\}$ and $T=\{X'_1,X'_2,X'_3\}$ where, for $i=1,2,3$, $C_i$ is the cell of $\cC$ that contains $Y_i$ and $X'_i$ is a random uniform point in $\bX'\cap C_i$. This definition of $E$ is only valid when $\{C_1,C_2\}$ is an edge of $\cC$. Otherwise $E=\emptyset$. Similarly, the definition of $T$ assumes $\{C_1,C_2,C_3\}$ is a triangle of $\cC$. Otherwise $T=\emptyset$.
\end{enumerate}
\end{RE}

This procedure has the same key properties as Random Experiment~\ref{REI} even though the triangles $\{Y_1,Y_2,Y_3\}$ are not equilateral. Specifically, each $Y_i$ is uniform on $\Tone$, the length of each edge $\{Y_i,Y_j\}$ is uniformly distributed on $[0,r]$ and, conditional on $T=t$ for some triangle $t$ of $\cG'$, $E$ chooses each one of the three edges of $t$ with probability $1/3$. As a result, Lemma~\ref{lem:ET} holds with the same proof (word for word) as in Section~\ref{sec:triangles}. As a result, we can define the same fractional matching $f$ as in equation~\eqref{eq:fdef} of that section, and then~\eqref{eq:fracmatch}, \eqref{eq:ftbound}, \eqref{eq:alphabound} work without changes.
Summarizing, for each triangle $t$ of $\cG'$, our fractional matching is defined by (cf.~\eqref{eq:fdef})
\begin{equation}
\label{eq:fdef3}
f(t) = \pr(T=t\mid E=e) = \frac{\theta r^d N^2}{6s^{2d}} \pr(T=t),
\end{equation}
where $e$ is any edge of $\cG'$ contained in $t$, and (for $\kappa_2$ sufficiently large given $\eps$) we can ensure the existence of a matching in $\cH'$ of size at least $(1-\eps)f(\cH')$.

The second ingredient that we need to modify from Section~\ref{sec:triangles} is the lower bound of $f(\cH')$. We will prove that Lemma~\ref{lem:fH} still works for the $d=1$ case with a slightly different constant. That is,
for $c=8$ and sufficiently small $\epsilon>0$,
\begin{equation}
\label{eq:fH3}
f(\cH') \ge (1-c\eps) \frac{1}{6}\theta r^d n^2.
\end{equation}
We will restate~\eqref{eq:fH3} as Lemma~\ref{lem:fH3} and prove it at the end of the section. 
As a conclusion, we have found a matching in $\cH'$ of size at least
\[
(1-\eps)f(\cH') \ge (1-(c+1)\eps) \theta r^d n^2/6,
\]
exactly as in~\eqref{eq:almostdone}. At this point, the proof of the theorem for $d=1$ proceeds identically to that in Section~\ref{sec:triangles} for the $d\ge2$ case. This finishes the proof of Theorem~\ref{thm:dense} for $d=1$, assuming the conclusions of Lemmas~\ref{lem:uniform} and~\ref{lem:fH3} which we will prove below.
\end{proof}

For a continuous random variable $X$, let $F_X,f_X$ denote its {\em cumulative distribution} and {\em probability density} functions, respectively (shortened to c.d.f.\ and~p.d.f.), which satisfy, for any $x\in\real$,
\[
\pr(X\le x) = F_X(x) = \int_{-\infty}^x f_X(t)~dt.
\]
\begin{lem}
\label{lem:uniform}
For $R_1,R_2,R_3,R$ defined as in step~\eqref{step:REIII1} of Random Experiment~\ref{REIII} and for any $0\le x\le 1$,
\[
F_{R_1}(rx)=F_{R_2}(rx)=\frac{3x-x^3}{2},
\qquad
F_{R_3}(rx) = x^3
\qquad\text{and}\qquad
F_R(rx) = x.
\]
In particular, $R$ is uniformly distributed on $[0,r]$.
\end{lem}
\begin{proof}
Let $x$ always denote an arbitrary real in $[0,1]$.
From the definition of $R_3$ as the maximum of three independent uniform random variables $U_1,U_2,U_3$ on $[0,r]$,
\[
F_{R_3}(rx) = \pr(U_1,U_2,U_3 \le rx) = x^3
\qquad\text{and}\qquad
f_{R_3}(rx) = \frac{d}{dx} F_{R_3}(rx) = 3x^2.
\]
From the definition of $R_1$,
\[
f_{R_1}(rx) = \int_x^1  \frac{1}{t}   f_{R_3}(rt)~dt  = \int_x^1 3t~dt = \frac{3}{2}(1-x^2)
\]
and thus
\[
F_{R_1}(rx) = \int_0^x f_{R_1}(rt)~dt = \frac{1}{2}(3x-x^3).
\]
By symmetry $R_1,R_2$ are identically distributed, so $F_{R_2}(rx)=F_{R_1}(rx)$ and $f_{R_2}(rx)=f_{R_1}(rx)$.
Finally,
\[
f_{R}(rx) = \frac{1}{3}(f_{R_1}(rx) + f_{R_2}(rx) + f_{R_3}(rx)) = 1,
\]
so
\[
F_{R}(rx) = \int_0^x f_{R}(rt)~dt = x,
\]
as desired. This concludes the proof.
\end{proof}

Finally we prove the upper bound announced in~\eqref{eq:fH3}, akin to Lemmas~\ref{lem:fH} and~\ref{lem:fH2} in earlier sections.
\begin{lem}
\label{lem:fH3}
For $c=8$ and sufficiently small $\epsilon>0$,
\[
f(\cH') \ge (1-c\eps) \frac{1}{6}\theta r^d n^2.
\]
\end{lem}
\begin{proof}
Let $T,E$ be generated as in Random Experiment~\ref{REIII}.
The strategy we followed in Lemmas~\ref{lem:fH} and~\ref{lem:fH2} was to show that for most edges $e$ in $\cG'$, $\pr(T\ne\emptyset\mid E=e)=1$. This is not true here. Instead, we will show that, for all edges $e$, $\pr(T\ne\emptyset\mid E=e)$ is close to $1$.

Let $R_1,R_2,R_3$ and $W_1,W_2,W_3$ be as in step~\eqref{step:REIII2} of Random Experiment~\ref{REIII}. Note that, if
\begin{equation}
\label{eq:Ribounds}
\eps r\le R_i \le (1-\eps) r
\qquad
\forall i=1,2,3,
\end{equation}
then the cells containing $W_1,W_2,W_3$ must be all different and pairwise adjacent in $\cC$, so $T\ne\emptyset$.
From Lemma~\ref{lem:uniform},
the probability that~\eqref{eq:Ribounds} fails due to $R_1$  (or symmetrically due to $R_2$) is at most
\begin{align*}
F_{R_1}(\eps r) + 1 - F_{R_1}((1-\eps) r) &= \frac{1}{2}(3\eps-\eps^3) + 1 - \frac{1}{2}(3(1-\eps)-(1-\eps)^3)
\\
&\le  \frac{1}{2}( 3\eps  + 3\eps^2 ) < 2\eps,
\end{align*}
where in the last inequality we assumed $\eps<1/3$.
Similarly, the probability that~\eqref{eq:Ribounds} fails due to $R_3$ is at most
\[
F_{R_3}(\eps r) + 1 - F_{R_3}((1-\eps) r) = \eps^3 + 1 - (1-\eps)^3 \le 3\eps + \eps^3 < 4\eps.
\]
As a result,
\[
\pr(T\ne\emptyset) \ge 1 - 2\eps - 2\eps - 4\eps = 1 - 8\eps.
\]
Finally, from~\eqref{eq:fdef3} and since $N^2/s^{2d}=N^2{n_0}^2\ge n^2$,
\[
f(\cH') = \sum_{t\in\cH'} f(t) = \sum_{t\in\cH'} \frac{\theta r^d N^2}{6s^{2d}} \pr(T=t) \ge \frac{\theta r^d n^2}{6} \pr(T\ne\emptyset) \ge  (1 - 8\eps)\frac{\theta r^d n^2}{6},
\]
as claimed.
\end{proof}

\section{Negative results. Proof of Theorem~\ref{thm:neg}}\label{sec:neg}

\begin{proof}
Our goal is to show that, for large enough $t$, a.a.s.\ we cannot find a $K_t$-packing in the random geometric graph $\cG_*(\bX,r)$ that uses more than roughly a $\rho$ fraction of the edges.

Throughout the proof, we will assume that $\kappa_3 \ge r=\omega(n^{-2/d})$ (so that we can apply Lemma~\ref{lem:edges}). For $r$ below that range, a.a.s.\ the random graph $\cG_*(\bX,r)$ has only a bounded number of edges and they induce a matching, so $\nu_t(\cG_*(\bX,r))=0$.

In the same spirit of Remark~\ref{rem:torus}, we first show that the statement for the toroidal version of the model ($*=\Td$) implies the statement for the cubic version ($*=\Qd$).
Indeed, suppose that the largest $K_t$-packing in $\cG_\Td(\bX,r)$ has size at most 
\[
(1+\delta/2) \frac{m(\cG_\Td(\bX,r))}{2\binom{t}{2}}.
\]
From Lemma~\ref{lem:edges}, we can assume that $m(\cG_\Qd(\bX,r)) \ge \frac{1+\delta/2}{1+\delta} m(\cG_\Td(\bX,r))$ by taking $\kappa_3$ to be sufficiently small given $d,\delta$. Hence, the largest $K_t$-packing in $\cG_\Td(\bX,r)$ has size at most 
\[
(1+\delta/2) \frac{m(\cG_\Td(\bX,r))}{2\binom{t}{2}} \le (1+\delta) \frac{m(\cG_\Qd(\bX,r))}{2\binom{t}{2}},
\]
so we obtain the conclusion of Theorem~\ref{thm:neg} for $*=\Qd$. In view of that, we will restrict our attention to the toroidal case $*=\Td$ for the remaining of the proof and in particular take $\kappa_3=1/2$ to guarantee that the edges of $\cG_\Td(\bX,r)$ do not wrap around the torus.

Let $0<\delta<1$, pick $\eps>0$ sufficiently small given $d,\delta$, and suppose $t_0$ is sufficiently large given $d,\delta,\eps$.
We partition the torus $[0,1]^d$ into $n_0$ cubic cells of side length $s=\lceil1/\eps r\rceil^{-1}$, so in particular $\eps r/(1+\eps r) \le s\le \eps r$ and $n_0=s^{-d}$ is an integer.
For later use, since $r\le\kappa_3=1/2$ and by picking $\eps$ small enough given $d,\delta$,
\begin{equation}
\label{eq:sbound}
s^d\ge (1-\eps/2)^d (\eps r)^d \ge \frac{(\eps r)^d}{1+\delta/10}.
\end{equation}
Now, an edge $X_iX_j$ of $\cG_\Td(\bX,r)$ is called {\em short} if both endpoints $X_i$ and $X_j$ belong to the same cell. 
\begin{claim}\label{clm:2}
A.a.s.\ at most a $(1+\delta/3)\eps^d/\theta$ fraction of the edges of $\cG_\Td(\bX,r)$ are short.
\end{claim}
The proof of the claim follows from a simple second moment method akin to that of Lemma~\ref{lem:edges}. Indeed, let $\hat Y$ be the number of short edges. Then,
\[
\ex \hat Y = \binom{n}{2} s^d \le \eps^d \binom{n}{2} r^d \sim \frac{\eps^d}{\theta} m,
\]
where in the last step we use the asymptotic estimate of $m$ from Lemma~\ref{lem:edges}.
Moreover, by analogy with~\eqref{eq:Y2moment},
\[
\ex \hat Y^2 = O(n^2s^d + n^3s^{2d}) + \binom{n}{2}\binom{n-2}{2} s^{2d} \sim (\ex \hat Y)^2.
\]
Hence, by the second moment method, a.a.s.\ $\hat Y \le (1+\delta/3) \frac{\eps^d}{\theta} m$.
This proves Claim~\ref{clm:2}.

Let $\theta=\theta_d = \frac{\pi^{d/2}}{\Gamma(d/2+1)}$ be the volume of a unit ball in $\real^d$ (with $\theta_0=1$ by convention), and let $\rho=\rho_d$ be the volume of the intersection of two unit balls with centers at distance $1$ divided by $\theta$.

\begin{claim}
\label{claim:rho}
$\rho=\rho_d \le \min\{ 1/2, \sqrt{d} (3/4)^{d/2}\}$ for all $d\ge1$.
\end{claim}
The claim is trivial for $d=1$, since $\rho_1=1/2$. To prove it for $d\ge2$, let $S$ be the intersection of two unit balls $B_1,B_2$ with centers $X_1,X_2$ at distance $1$, so in particular $S$ has volume $\rho\theta$.
Clearly, $S$ is contained in the one half of $B_1$ that is closer to $X_2$, so we get the first bound $\rho \le 1/2$.
W.l.o.g.\ assume $X_1=(0,0,\ldots,0)$ and $X_2=(1,0,\ldots,0)$. For $i=1,2$, let $S_i$ be set of points in $S$ that are closer to $X_i$. Then, $\vol(S)=\vol(S_1)+\vol(S_2)$ and $\vol(S_1)=\vol(S_2)$, so $\vol(S_i)=\rho\theta/2$. Moreover, a point $X=(x_1,\ldots,x_d)\in\real^d$ belongs to $S_2$ if and only if $1/2\le x_1\le 1$ and ${x_2}^2+\cdots+{x_d}^2 \le 1-{x_1}^2$, so in particular $X$ belongs to a $(d-1)$-ball of radius $\sqrt{1-{x_1}^2}$. Hence,
\[
\vol(S_2) = \int_{1/2}^1 \theta_{d-1} (1-t^2)^{(d-1)/2}~dt \le \frac{\theta_{d-1}}{2}  \left(\frac{3}{4}\right)^{(d-1)/2},
\]
where the last inequality follows from bounding the integrand by $\theta_{d-1} (3/4)^{(d-1)/2}$, and then
\begin{equation}
\label{eq:rhobound1}
\rho_d = 2\vol(S_2)/\theta_d \le \frac{\theta_{d-1}}{\theta_d}  \left(\frac{3}{4}\right)^{(d-1)/2}.
\end{equation}
Now
\begin{equation}
\label{eq:rhobound2}
\frac{\theta_{d-1}}{\theta_d} = \pi^{-1/2} \frac{\Gamma(d/2+1)}{\Gamma(d/2+1/2)}
\le \sqrt{\frac{d+1}{2\pi}} \le \sqrt{\frac{3d}{4}},
\end{equation}
where in the first inequality above we used a standard bound on $\Gamma(d/2+1)/\Gamma(d/2+1/2)$ from~\cite{Kershaw}. Combining~\eqref{eq:rhobound1} and~\eqref{eq:rhobound2} gives the second bound $\rho\le\sqrt{d} (3/4)^{d/2}$, and finishes the proof of the claim.

Next we give a deterministic result that asserts that any drawing of $K_t$ in the torus with vertices at distance at most $r$ must have many short edges. The proof is deferred to the end of the section.

\begin{claim}\label{clm:3}
Let $t\ge t_0$. For any clique $K_t$ with vertices in $\Td$ where every two vertices have distance at most $r$, at least an $\frac{\eps^d}{(1+\delta/3)\rho\theta}$ fraction of its edges are short.
\end{claim}

Now assume that Claim~\ref{clm:3} and the a.a.s.\ conclusion of Claim~\ref{clm:2} hold.
Then, $\nu_t=\nu_t(\cG_\Td(\bX,r))$ and $m=m(\cG_\Td(\bX,r))$ must satisfy
\[
\frac{\eps^d}{(1+\delta/3)\rho\theta}\binom{t}{2} \nu_t \le (1+\delta/3)\eps^d m/\theta,
\]
since the number of short edges in the largest $K_t$-packing must be at most the total number of short edges.
Hence, since $\delta<1$,
\[
\nu_t \le (1+\delta/3)^2 \rho \cdot \frac{m}{\binom{t}{2}} \le (1+\delta) \rho \cdot \frac{m}{\binom{t}{2}},
\]
which, together with the bounds on~$\rho$ from Claim~\ref{claim:rho}, completes the proof of Theorem~\ref{thm:neg}.

It only remains to prove Claim~\ref{clm:3}. To do so, pick any $K_t$ with vertices in $\Td$ and edges of length at most $r$. Since $r\le\kappa_3=1/2$, edges do not wrap around the torus, so we may as well assume for simplicity that the vertices are in $\real^d$. Let $v$ and $w$ be two vertices of $K_t$ that are furthest apart. Suppose they are at distance $x\le r$, and let $S(v,w)$ be the set of points in $\real^d$ that are at distance at most $x$ from both $v$ and $w$, i.e.~$S(v,w)$ is simply the intersection of two balls of radius $x$ centered at $v$ and $w$. All vertices of $K_t$ are contained in $S(v,w)$ by construction. We can assume w.l.o.g.\ that $x=r$ (otherwise we can replace $v,w$ by  some points $v',w'$ in $\real^d$ at distance $r$ with $S(v',w')\supset S(v,w)$). Summarizing, all the vertices of $K_t$ lie inside a set $S=S(v,w)$, which is the intersection of two balls $B(v)$ and $B(w)$ of radius $r$ with centers at distance $r$. So $\vol(S)=\rho\theta r^d$, by definition of $\rho$ and rescaling by a factor of $r$.
Let $\hat S$ be the union of all cells that intersect $S$. By construction, $\hat S$ is contained in the union of $S$, $\hat S_v$ and $\hat S_w$, where $\hat S_v,\hat S_w$ are the sets of all cells that intersect the boundaries of $B(v)$ and $B(w)$, respectively. Since cells have side $s\le\eps r$, the ones that intersect the boundary of $B(v)$ are fully contained in an annulus of center $v$ and radii $(1-d\eps)r$ and $(1+d\eps)r$ (with room to spare).
As a result,
\[
\vol(\hat S_v) = \vol(\hat S_w) \le \theta(1+d\eps)^dr^d - \theta(1-d\eps)^dr^d \le (2d^2\eps)(1+d\eps)^{d-1}\theta r^d,
\]
where we used the bound $(1+z)^d-(1-z)^d\le 2dz(1+z)^{d-1}$, which holds for all real $z\ge0$ and $d\in\nat$ (and follows easily from the Mean Value Theorem). By choosing $\eps$ sufficiently small given $d,\delta$, we can assume that $\vol(\hat S_v), \vol(\hat S_w) \le \delta\rho \theta r^d/20$, so
\[
\vol(\hat S) \le \vol(S)+\vol(\hat S_v)+\vol(\hat S_w) \le (1+\delta/10)\rho\theta r^d.
\]
Now let $k$ be the number of cells contained in $\hat S$. From the above and~\eqref{eq:sbound}, we must have
\begin{equation}
\label{eq:kbound}
k = \vol(\hat S)/s^d \le (1+\delta/10) \vol(\hat S) / (\eps r)^d \le
(1+\delta/10)^2\rho\theta / \eps^d.
\end{equation}
Since the upper bound on $k$ only depends on $d,\delta,\eps$, we can assume from our choice of $t_0$ that
\begin{equation}
\label{eq:kt}
(1-k/t) \ge 1/(1+\delta/10).
\end{equation}
Label the cells in $\hat S$ with numbers from $1$ to $k$. For $i\in[k]$, let $t_i$ denote the number of
vertices of $K_t$ contained in cell $i$. Recall that, by construction, there are no vertices of $K_t$ outside of $\hat S$. So $t_1 + \cdots+ t_k = t$. Then, the number of short edges is (explanation follows)
\[
\sum_{i=1}^k \binom{t_i}{2} \ge k \binom{t/k}{2} =  \frac{t^2}{2k} (1-k/t) \ge
\frac{t^2}{2(1+\delta/10)^3\rho\theta / \eps^d}
\ge \frac{\eps^d}{(1+\delta/3)\rho\theta} \binom{t}{2}.
\]
The first inequality above follows from the convexity of the function $\binom{z}{2}=\frac{z(z-1)}{2}$, the second inequality combines~\eqref{eq:kbound} and~\eqref{eq:kt}, and the third one uses $\delta<1$. This yields Claim~\ref{clm:3} and completes the proof of Theorem~\ref{thm:neg}.
\end{proof}

\begin{remark}\label{rem:constants}
We do not attempt to optimize the constants in the proof of Theorem~\ref{thm:neg}. However, a closer inspection of the argument easily shows that for fixed $\delta$ (say, $\delta=1/2$), one can pick $\eps = 1/d^{O(1)}$ and $t_0=(1/\eps)^{O(d)} = d^{O(d)}$ (as $d\to\infty$) that work.

\end{remark}

\section{Concluding remarks}\label{sec:final}

The obvious open problem is to ``close the gap'' in Corollary~\ref{cor:Tuza}.
\begin{question}
Does Tuza's conjecture hold for the random geometric graph $\cG_*(\bX,r)$ in the regime $\kappa_1 n^{-1/d} \le r \le \kappa_2 n^{-1/d}$?
\end{question}
Theorems~\ref{thm:unitdistance} and~\ref{thm:neg} raise another interesting question which we now discuss. Note that, if $\chi(F) \le k$, then $F$ is contained in a blowup of $K_k$, which is a distance graph in $\mathbb{R}^{k-1}$. Thus $F \in \cU(k-1)$. By Theorem~\ref{thm:unitdistance}, for any dimension $d \ge \chi(F)-1$ and $\kappa_2n^{-1/d} \le  r \le \kappa_3$, a.a.s.\ there is an almost-perfect covering of the edges of $\cG_*(\bX,r)$ with edge-disjoint copies of $F$. 
It is then natural to define $d_0(F)$ to be the smallest value of $d$ for which the random geometric graph $\cG_*(\bX,r)$ admits an almost-perfect $F$-packing a.a.s.~for a ``reasonable'' range of $r$.
By our discussion above, for all $F$ we have the upper bound 
\[
d_0(F) \le \chi(F)-1,
\]
but in fact we can do better for the triangle $K_3$ since we showed in Theorem~\ref{thm:dense} that $d_0(K_3) = 1 = \chi(K_3)-2$.
By Theorem~\ref{thm:neg} and Remark~\ref{rem:constants}, we have the lower bound 
\[
d_0(K_t) = \Omega\rbrac{\frac{\log t}{\log \log t}}
\qquad\text{as }
t\to\infty.
\]
Thus, we ask the following. 
\begin{question}
    Can we determine $d_0(F)$, or at least find good estimates?
\end{question}
\noindent
For example, we have $d_0(K_4) \le \chi(K_4)-1 = 3$. We suspect that $d_0(K_4) = 3$, but it could certainly be smaller as in the triangle case.

\section{Acknowledgements}

The authors would like to thank Alberto Espuny D\'iaz for calling their attention to the fact that Theorem \ref{thm:dense} holds for $d=1$ (the proof we give in this paper is our own). We would also like to thank Alexandra Wesolek for calling our attention to Lucas Burggraf's thesis \cite{Burggraf}.

\bibliographystyle{abbrv}
\bibliography{refs}

@article {Yuster,
    AUTHOR = {Yuster, Raphael},
     TITLE = {Dense graphs with a large triangle cover have a large triangle
              packing},
   JOURNAL = {Combin. Probab. Comput.},
  FJOURNAL = {Combinatorics, Probability and Computing},
    VOLUME = {21},
      YEAR = {2012},
    NUMBER = {6},
     PAGES = {952--962},
      ISSN = {0963-5483,1469-2163},
   MRCLASS = {05C70},
  MRNUMBER = {2981163},
MRREVIEWER = {Shiying\ Wang},
       DOI = {10.1017/S0963548312000235},
       URL = {https://doi.org/10.1017/S0963548312000235},
}

@article {BK,
    AUTHOR = {Baron, Jacob D. and Kahn, Jeff},
     TITLE = {Tuza's conjecture is asymptotically tight for dense graphs},
   JOURNAL = {Combin. Probab. Comput.},
  FJOURNAL = {Combinatorics, Probability and Computing},
    VOLUME = {25},
      YEAR = {2016},
    NUMBER = {5},
     PAGES = {645--667},
      ISSN = {0963-5483,1469-2163},
   MRCLASS = {05C70 (05C50)},
  MRNUMBER = {3531437},
MRREVIEWER = {Steven\ D.\ Noble},
       DOI = {10.1017/S0963548316000067},
       URL = {https://doi.org/10.1017/S0963548316000067},
}

@book{JLR,
  title={Random Graphs},
  author={Janson, Svante and \L uczak, Tomasz and Ruci\'nski, Andrzej},
  series={Wiley Series in Discrete Mathematics and Optimization},
  year={2011},
  publisher={Wiley}
}

@article {K96,
    AUTHOR = {Kahn, Jeff},
     TITLE = {A linear programming perspective on the
              {F}rankl-{R}\"odl-{P}ippenger theorem},
   JOURNAL = {Random Structures Algorithms},
  FJOURNAL = {Random Structures \& Algorithms},
    VOLUME = {8},
      YEAR = {1996},
    NUMBER = {2},
     PAGES = {149--157},
      ISSN = {1042-9832,1098-2418},
   MRCLASS = {05C65 (90C05)},
  MRNUMBER = {1607104},
MRREVIEWER = {Daniel\ Turz\'ik},
       DOI = {10.1002/(sici)1098-2418(199603)8:2<149::aid-rsa5>3.0.co;2-y},
       URL =
              {https://doi.org/10.1002/(sici)1098-2418(199603)8:2<149::aid-rsa5>3.0.co;2-y},
}

@article {BDZ,
    AUTHOR = {Bennett, Patrick and Dudek, Andrzej and Zerbib, Shira},
     TITLE = {Large triangle packings and {T}uza's conjecture in sparse
              random graphs},
   JOURNAL = {Combin. Probab. Comput.},
  FJOURNAL = {Combinatorics, Probability and Computing},
    VOLUME = {29},
      YEAR = {2020},
    NUMBER = {5},
     PAGES = {757--779},
      ISSN = {0963-5483,1469-2163},
   MRCLASS = {05B40 (05C80 05D40)},
  MRNUMBER = {4152570},
MRREVIEWER = {Mehrdad\ Nasernejad},
       DOI = {10.1017/s0963548320000115},
       URL = {https://doi.org/10.1017/s0963548320000115},
}

@article {BCD,
    AUTHOR = {Bennett, Patrick and Cushman, Ryan and Dudek, Andrzej},
     TITLE = {Closing the random graph gap in {T}uza's conjecture through
              the online triangle packing process},
   JOURNAL = {SIAM J. Discrete Math.},
  FJOURNAL = {SIAM Journal on Discrete Mathematics},
    VOLUME = {35},
      YEAR = {2021},
    NUMBER = {3},
     PAGES = {2145--2169},
      ISSN = {0895-4801,1095-7146},
   MRCLASS = {05C70 (05B40 05C80 05D40)},
  MRNUMBER = {4313839},
MRREVIEWER = {Mehrdad\ Nasernejad},
       DOI = {10.1137/20M1351771},
       URL = {https://doi.org/10.1137/20M1351771},
}

@article{DR88,
author = {Dubhashi, Devdatt and Ranjan, Desh},
title = {Balls and bins: A study in negative dependence},
journal = {Random Structures \& Algorithms},
volume = {13},
number = {2},
pages = {99-124},
year = {1998}
}

@article {KP,
    AUTHOR = {Kahn, Jeff and Park, Jinyoung},
     TITLE = {Tuza's conjecture for random graphs},
   JOURNAL = {Random Structures Algorithms},
  FJOURNAL = {Random Structures \& Algorithms},
    VOLUME = {61},
      YEAR = {2022},
    NUMBER = {2},
     PAGES = {235--249},
      ISSN = {1042-9832,1098-2418},
   MRCLASS = {05C80 (05C70 60C05)},
  MRNUMBER = {4456027},
MRREVIEWER = {Mehrdad\ Nasernejad},
       DOI = {10.1002/rsa.21057},
       URL = {https://doi.org/10.1002/rsa.21057},
}

@article{Kershaw,
 author = {Kershaw, Donald},
 journal = {Mathematics of Computation},
 number = {164},
 pages = {607--611},
 publisher = {American Mathematical Society},
 title = {Some Extensions of {W}.~{G}autschi's Inequalities for the {G}amma Function},
 volume = {41},
 year = {1983}
}

@book{Penrose,
  title={Random Geometric Graphs},
  author={Penrose, Mathew},
  series={Oxford studies in probability},
  year={2003},
  publisher={Oxford University Press}
}

@article {tuza2,
    AUTHOR = {Tuza, Zsolt},
     TITLE = {A conjecture on triangles of graphs},
   JOURNAL = {Graphs Combin.},
  FJOURNAL = {Graphs and Combinatorics},
    VOLUME = {6},
      YEAR = {1990},
    NUMBER = {4},
     PAGES = {373--380},
      ISSN = {0911-0119,1435-5914},
   MRCLASS = {05C35},
  MRNUMBER = {1092587},
MRREVIEWER = {Mario\ Gionfriddo},
       DOI = {10.1007/BF01787705},
       URL = {https://doi.org/10.1007/BF01787705},
}

@proceedings {tuza1,
     TITLE = {Finite and infinite sets. {V}ol.\ {I}, {II}},
    SERIES = {Colloquia Mathematica Societatis J\'anos Bolyai},
    VOLUME = {37},
 BOOKTITLE = {Proceedings of the sixth {H}ungarian combinatorial colloquium
              held in {E}ger, {J}uly 6--11, 1981},
    EDITOR = {Hajnal, A. and Lov\'asz, L. and S\'os, V. T.},
 PUBLISHER = {North-Holland Publishing Co., Amsterdam},
      YEAR = {1984},
     PAGES = {902},
      ISBN = {0-444-86763-5},
   MRCLASS = {05-06 (04-06)},
  MRNUMBER = {818224},
}

@article {haxell,
    AUTHOR = {Haxell, P. E.},
     TITLE = {Packing and covering triangles in graphs},
   JOURNAL = {Discrete Math.},
  FJOURNAL = {Discrete Mathematics},
    VOLUME = {195},
      YEAR = {1999},
    NUMBER = {1-3},
     PAGES = {251--254},
      ISSN = {0012-365X,1872-681X},
   MRCLASS = {05C70},
  MRNUMBER = {1663859},
       DOI = {10.1016/S0012-365X(98)00183-6},
       URL = {https://doi-org.libproxy.library.wmich.edu/10.1016/S0012-365X(98)00183-6},
}

@article {az,
    AUTHOR = {Aharoni, Ron and Zerbib, Shira},
     TITLE = {A generalization of {T}uza's conjecture},
   JOURNAL = {J. Graph Theory},
  FJOURNAL = {Journal of Graph Theory},
    VOLUME = {94},
      YEAR = {2020},
    NUMBER = {3},
     PAGES = {445--462},
      ISSN = {0364-9024,1097-0118},
   MRCLASS = {05C65 (05C70)},
  MRNUMBER = {4102610},
MRREVIEWER = {Mehrdad\ Nasernejad},
       DOI = {10.1002/jgt.22533},
       URL = {https://doi-org.libproxy.library.wmich.edu/10.1002/jgt.22533},
}

@article {hr,
    AUTHOR = {Haxell, P. E. and R\"odl, V.},
     TITLE = {Integer and fractional packings in dense graphs},
   JOURNAL = {Combinatorica},
  FJOURNAL = {Combinatorica. An International Journal on Combinatorics and
              the Theory of Computing},
    VOLUME = {21},
      YEAR = {2001},
    NUMBER = {1},
     PAGES = {13--38},
      ISSN = {0209-9683,1439-6912},
   MRCLASS = {05C70},
  MRNUMBER = {1805712},
MRREVIEWER = {R.\ Balakrishnan},
       DOI = {10.1007/s004930170003},
       URL = {https://doi-org.libproxy.library.wmich.edu/10.1007/s004930170003},
}

@article {krivelevich,
    AUTHOR = {Krivelevich, Michael},
     TITLE = {On a conjecture of {T}uza about packing and covering of
              triangles},
   JOURNAL = {Discrete Math.},
  FJOURNAL = {Discrete Mathematics},
    VOLUME = {142},
      YEAR = {1995},
    NUMBER = {1-3},
     PAGES = {281--286},
      ISSN = {0012-365X,1872-681X},
   MRCLASS = {05C70},
  MRNUMBER = {1341453},
       DOI = {10.1016/0012-365X(93)00228-W},
       URL = {https://doi-org.libproxy.library.wmich.edu/10.1016/0012-365X(93)00228-W},
}

@misc{Burggraf,
  author = {Burggraf, Lucas and Wesolek, Alexandra},
  howpublished = "personal communication",
  year = "2025",
}

\end{document}